\documentclass[preprint,12pt]{article}
\usepackage{amssymb}
\usepackage{mathrsfs}
\usepackage{amsfonts}
\usepackage{graphicx}
\usepackage{colortbl,dcolumn}
\usepackage{amsmath}
\usepackage{psfrag}
\usepackage{booktabs}
\usepackage{subfig}
\usepackage{cite}
\numberwithin{equation}{section}
\usepackage[top=1.2in, bottom=1.2in, left=1in, right=1in]{geometry}

\newtheorem{theorem}{Theorem}[section]
\newtheorem{proof}{Proof}
\newtheorem{definition}{Definition}
\newtheorem{lemma}{Lemma}

\begin{document}
%\title{On weak error analysis of the exponential Euler
%scheme for stochastic partial differential equation with
%multiplicative noise \footnotemark[1]}
\title{Exponential Discrete Gradient Schemes for Stochastic Differential Equations \footnotemark[1]}
       \author{
        Jialin Ruan\footnotemark[2], Lijin Wang\footnotemark[3]\\
      {\small\footnotemark[2]~\footnotemark[3] School of Mathematical Sciences, }\\{\small University of Chinese Academy of Sciences, }\\
         {\small  Beijing 100049, China }} 
       %{\small \footnotemark[3] School of Mathematical Sciences, University of Chinese Academy of Sciences,}\\
    %{\small Beijing 100190, P.R.China }}
       \maketitle
       \footnotetext{\footnotemark[1] Supported by NNSFC No. 11071251, No.11471310}
        %\footnotemark[3]The second author is supported by the NNSFC (No. 11071251). \\\footnotemark[1]The fourth author is supported by the NNSFC (NO. 11301001, NO. 2013SQRL030ZD).}
        \footnotetext{\footnotemark[3] Corresponding author: ljwang@ucas.ac.cn}
         % \footnotetext{}  \\
%               \footnotetext{\footnotemark[1] }
%                office space and partially organizing
%                the author's short visit to ETH
%                Z\"{u}rich in 2013.

       \begin{abstract}
          {\rm\small In this paper, we propose a class of stochastic exponential discrete gradient schemes for SDEs with linear and gradient components in the coefficients. The root mean-square errors of the schemes are analyzed, and the structure-preserving properties of the schemes for SDEs with special structures are investigated. Numerical tests are performed to verify the theoretical results and illustrate the numerical behavior of the proposed methods.
}\\

\textbf{AMS subject classification: } {\rm\small 65C30, 60H10, 65D30.}\\
%\textbf{PACS: 02.60.Lj}

\textbf{Key Words: }{\rm\small} stochastic differential equations; exponential integrators; discrete gradient methods; mean-square convergence; structure-preserving algorithms.\end{abstract}
\section{Introduction}
\label{sec:intro}

We consider the following SDE containing a linear (L) part in the drift term and gradient (G) parts in the drift and diffusion coefficients,
\begin{equation}\label{intro:eq-6}
  {\rm d}X(t) = (A X(t) + Q_1 \nabla U(X(t))){\rm d}t + \sum_{r=1}^mQ_{2 ,r}\nabla V_r(X(t)) \circ {\rm d}W_r(t), \quad X(t_0) = x_0,
\end{equation}
where $A$, $Q_1$ and $Q_{2,r}$ ($r=1,\dots,m$) are $d \times d$  real matrices, and $U, V_r: \mathbb{R}^d \to \mathbb{R}$ (r=1,\dots,m) are differentiable functions, and $\mathcal{W}(t)=(W_1(t),\dots, W_m(t))$ is an m-dimensional standard Wiener process. We call (\ref{intro:eq-6}) a L-G SDE in this paper.

Many important stochastic systems can be written as the L-G SDEs (\ref{intro:eq-6}),  such as the $d=2\bar{d}$-dimensional stochastic Hamiltonian systems (SHSs) (see e.g. \cite{milstein2002symplectic},\cite{milstein2002numerical})
\begin{equation}\label{intro:eq-7}
  {\rm d}X(t) = J^{-1}( M X(t) + \nabla U(X(t))){\rm d}t + J^{-1} \sum_{r=1}^m\nabla V_r(X(t)) \circ {\rm d}W_r(t), \quad X(t_0) = x_0,
\end{equation}
where $M$ is symmetric, $J=\begin{pmatrix}
      0 & I_{\bar{d}} \\
      -I_{\bar{d}} & 0
      \end{pmatrix}$, and the stochastic Langevin-type equations (see e.g. \cite{hong2017high},\cite{milstein2003quasi})
\begin{equation}\label{intro:eq-2}
\begin{split}
  {\rm d}P ={} & f_1 (Q) {\rm d}t - \nu \Gamma P {\rm d}t + \sum_{r=1}^{m} \sigma_{r} \circ {\rm d}W_{r}(t), \quad P(t_0) = p_0, \\
  {\rm d}Q ={} & M^{-1} P {\rm d}t, \quad Q(t_0) = q_0,
  \end{split}
\end{equation}
where $f_1(Q)=\nabla U_0(Q)$ for a scalar unction $U_0(Q)$, $\Gamma$ is an $\bar{d} \times \bar{d}$-dimensional constant matrix, $\sigma_r\in \mathbb{R}^{\bar{d}}$, $\nu \ge 0$, and $M$ an $\bar{d}\times \bar{d}$-dimensional positive definite matrix. Note that (\ref{intro:eq-2}) can be written in the form of (\ref{intro:eq-6}) when $A = \begin{pmatrix}
      - \nu \Gamma & 0 \\
      M^{-1} & 0
      \end{pmatrix}$, $Q_1=J^{-1}$, $U=U_0$, $Q_2=I_{2\bar{d}}$, and $V_r=\sigma_r P$ ($r=1,\dots, m$).

For numerical approximations of the L-G SDE (\ref{intro:eq-6}), on one hand, the linear component $A X(t)$  motivates the idea of using an analog of the exponential integrators for ODEs, which are characterized by including the calculation of matrix exponentials, and integrating the system (\ref{intro:eq-6}) exactly when $U,V_r=0$ $(r=1,\dots,m)$. Such integrators are designed for ODEs with high accuracy, especially for very stiff ODEs such as highly oscillatory problems, where the exponential integrators allow much larger time step sizes than non-exponential ones (see e.g. \cite{li2016exponential} and references therein). On the other hand, the gradients in the L-G SDEs (\ref{intro:eq-6}) usually underlie structural properties of the systems, for instance,  the stochastic Hamiltonian systems (\ref{intro:eq-7}) possess the symplectic structure (\cite{milstein2002symplectic}), and the Langevin-type equations (\ref{intro:eq-2}) have the conformal symplectic structure (\cite{hong2017high},\cite{milstein2003quasi}). It is then natural to pursue numerical methods that preserve the structures of the original systems, the so-called structure-preserving algorithms (\cite{Hairer2006Geometric}). The need of such structure-preserving numerical methods arises in astronomy, mechanics, molecular dynamics, and so on. For instance, in astronomy, structure preservation can imply that the computed trajectory of a celestial body does not deviate much from its true orbit even after a very long time simulation (see e.g. \cite{Hairer2006Geometric}). Typical examples of structure-preserving methods include symplectic methods for stochastic Hamiltonian systems (see e.g. \cite{milstein2002symplectic},\cite{milstein2002numerical},\cite{wangbit} and references therein), energy-preserving methods for stochastic systems with invariant energy (see e.g. \cite{cohen1}), etc.. To the best of our knowledge, however, most exponential integrators in literature up to now are not structure-preserving, except for certain special ones (see e.g. \cite{li2016exponential} and references therein).

The aim of the paper is to construct stochastic exponential discrete gradient (SEDG) schemes for L-G SDEs. As mentioned above, the exponential integrators are designed adapting to the linear parts, and the DG integrators are used for discretizing the gradient components of the L-G SDEs. Our goal is to obtain stochastic exponential integrators with good structure-preserving behavior, which extends relative work on exponential integrators in deterministic case for ODEs (e.g. \cite{li2016exponential}) to stochastic context.

Specifically, for the L-G SDE (\ref{intro:eq-6}), we shall show that the proposed SEDG methods are generally of root mean-square convergence order 1. For a class of highly oscillatory nonlinear stochastic Hamiltonian systems, we show the root mean-square errors of the SEDG methods decay with the increase of the oscillating frequency. This is in contrast with some of the standard numerical methods for SHSs such as the symplectic Euler-Maruyama scheme, where the errors grow with the increase of the oscillating frequency. Moreover, we prove the exact preservation of the symplecticity, as well as the linear growth of the expectation of the energy by the SEDG method for a highly oscillatory stochastic Hamiltonian system. For a class of stochastic Poisson systems with invariant energy, we show the SEDG methods can preserve the energy exactly, and for stochastic Langevin-type equations, we prove that the proposed SEDG methods can nearly preserve the conformal symplecticity within error of root mean-square order 2.

 The contents of the paper are organized as follows. In section \ref{sec:const},  we use the analog of the variation-of-constants formula to reformulate the L-G SDE (\ref{intro:eq-6}) in integral form, based on which we construct the SEDG scheme.  Then we analyze its accuracy in the mean-square sense. In section \ref{sec:struct}, we apply the SEDG methods to L-G SDEs with special structures, including a class of highly oscillatory SDEs, stochastic Poisson systems with invariant energy, and stochastic Langevin-type equations. The accuracy and structure-preserving properties of the SEDG methods for these systems are investigated. Numerical experiments are performed in section \ref{sec:numer} on several stochastic models to verify the theoretical analysis and illustrate the numerical behavior of the SEDG methods. A brief conclusion is given in section \ref{sec:concl}.
%------------------------------------------------

% section 2
%------------------------------------------------

\section{SEDG method for the general L-G SDE}
\label{sec:const}
Let $(\Omega, \mathscr{F}, \{\mathscr{F}_t\}_{t \geq 0}, P)$ be the underlying probability space of the L-G SDE (\ref{intro:eq-6})
with the filtration $\{\mathscr{F}_t\}_{t \geq 0}$, and $W_r(t)$ (r=1,\dots,m) be $\mathscr{F}_t$-adapted.
Denote $$Q_1 \nabla U(X)=:f(X), \quad Q_{2,r} \nabla V_r(X)=:g_r(X),\,\,\, r=1,\dots, m,$$
and rewrite \ref{intro:eq-6} into its equivalent It$\hat{o}$ SDE (see e.g. \cite{kloeden1992numerical})
\begin{equation}\label{const:eq-1}
  {\rm d}X(t) = (A X(t) + \bar{f}(X(t))){\rm d}t + \sum_{r=1}^m g_r(X(t)) {\rm d}W_r(t), \quad X(t_0) = x_0,
\end{equation}
where $\bar{f}(X(t))) = f(X(t)) + \frac{1}{2} \sum_{r=1}^m\frac{\partial g_r}{\partial x}(X(t)) g_r(X(t))$. To guarantee the existence and uniqueness of the solution of (\ref{intro:eq-6}) (see e.g. \cite{oksendal2013stochastic}),
assume that $f(x):\, \mathbb{R}^d\rightarrow \mathbb{R}^d$ and $g_r(x): \mathbb{R}^d\rightarrow \mathbb{R}^d$ $(r=1,\dots,m)$ are measurable functions such that  \begin{equation}\label{addeu}
 \begin{split}
   (i)\,\,\,\,& |\bar{f}(x)-\bar{f}(y)|+\sum_{r=1}^m |g_r(x)-g_r(y)|\\
   \,\,&\le L|x-y|,\quad \mbox{for all }\,\, x,y\in \mathbb{R}^d,\\
 (ii)\,\,\,& |\bar{f}(x)|^2+\sum_{r=1}^m |g_r(x)|^2\le L(1+|x|^2), \,\,\mbox{for all}\,\, x\in \mathbb{R}^d,
 \end{split}
  \end{equation}
  for certain $L>0$, moreover,
  \begin{equation}\label{addeu2}
 (iii)\,\,E|x_0|^2<\infty,
  \end{equation}
  where $X_0$ is independent of the $\sigma$-algebra generated by the $m$-dimensional Brownian motion $\mathcal{W}(t)$ for $t\ge 0$.

For convenience,  and without loss of generality, we first restrict ourself to
the following L-G SDE with an one-dimensional noise, i.e., the L-G SDE (\ref{intro:eq-6}) when $m=1$,
\begin{equation}\label{intro:eq-6-1}
 {\rm d}X(t) = (A X(t) + Q_1 \nabla U(X(t))){\rm d}t + Q_{2}\nabla V(X(t)) \circ {\rm d}W(t), \quad X(t_0) = x_0,
 \end{equation}
and then generalize the results to the system (\ref{intro:eq-6}) with $m>1$. To unify the notations for (\ref{intro:eq-6}) and (\ref{intro:eq-6-1}), let $$ Q_{2,1}=Q_2,\quad V_1=V,\quad \mbox{and}\,\,\,\, g_1=g=Q_2\nabla V.$$

\subsection{Variation-of-constants formula and discrete gradients}
Consider the change of variable $Z(t) = \exp(-At) X(t)$ for the L-G SDE (\ref{intro:eq-6-1}).
According to the Stratonovich chain rule (\cite{evans2012introduction}), we have
\begin{equation*}
 \begin{split}
  {\rm d}Z(t) ={} & -A\exp(-At) X(t) {\rm d}t + \exp(-At){\rm d}X(t) \\
  	          ={} & \exp(-At) Q_1 \nabla U(X(t)){\rm d}t + \exp(-At) Q_2 \nabla V(X(t)) \circ {\rm d}W(t).
 \end{split}
\end{equation*}

Then we can reformulate (\ref{intro:eq-6-1}) in integral form on $[t_0,t_0+h]$ for any $h\ge 0$ as follows, which can be seen as the stochastic analog of the variation-of-constants formula for ODEs,
\begin{equation}\label{const:eq-2}
\begin{split}
X(t_0+h) & ={} \exp(A h) X(t_0) + \int_{t_0}^{t_0+h} \exp(A (t_0+h-s)) Q_1 \nabla U(X(s)) {\rm d}s \\
                & + \int_{t_0}^{t_0+h} \exp(A (t_0+h-s)) Q_2 \nabla V(X(s)) \circ {\rm d}W(s).
\end{split}
\end{equation}
Our SEDG scheme will be constructed based on this formulation. Next we introduce the concept of discrete gradients.
\begin{definition}\label{const:de-2}
  For a differentiable function $H(y)$, $\bar{\nabla} H(y,\hat{y})$ is said to be
  a discrete gradient of $H(y)$ if it is continuous and satisfies:
  \begin{equation*}
  \left\{
   \begin{aligned}
    \bar{\nabla} & H(y,\hat{y})^{\rm T}(\hat{y}-y) ={} H(\hat{y}) - H(y), \\
    \bar{\nabla} & H(y,y) ={} \nabla H(y). \\
   \end{aligned}
  \right.
  \end{equation*}
\end{definition}
Furthermore, if $\bar{\nabla} H(y,\hat{y}) = \bar{\nabla} H(\hat{y},y)$ holds, it is called a symmetric discrete gradient (SDG) (\cite{hong2011discrete}).

Choose an ordering of the coordinates $y^i$ such as $y^1,y^2,\dots,y^d$. The  coordinate increment discrete gradient (\cite{mclachlan1999geometric}), denoted with  $\bar{\nabla} H_0 (y,\hat{y})$, is defined as
\begin{equation}\label{const:eq-3}
  \bar{\nabla} H_0(y,\hat{y}) := \begin{pmatrix}
      \frac{H(\hat{y}^1,y^2,\dots,y^d)-H(y^1,y^2,\dots,y^d)}{\hat{y}^1-y^1} \\
      \frac{H(\hat{y}^1,\hat{y}^2,\dots,y^d)-H(\hat{y}^1,y^2,\dots,y^d)}{\hat{y}^2-y^2} \\
      \vdots \\
      \frac{H(\hat{y}^1,\hat{y}^2,\dots,\hat{y}^d)-H(\hat{y}^1,\hat{y}^2,\dots,\hat{y}^{d-1},y^d)}{\hat{y}^d-y^d}
      \end{pmatrix}.
\end{equation}
Then it is not difficult to see that,
\begin{equation}\label{sdg1}
  \bar{\nabla} H(y,\hat{y}) := \frac{1}{2} (\bar{\nabla} H_0 (y,\hat{y})+\bar{\nabla} H_0 (\hat{y},y))
\end{equation}
gives a symmetric discrete gradient. Throughout the paper, we will use (\ref{sdg1}) as our SDG.

\subsection{The SEDG scheme}
 Let $0=t_0<t_1<\cdots<t_n<\cdots<t_N=T$ be an equidistant time discretization of
the time interval $[0,T]$ with step size $h$, i.e., $t_n = n h$, $n=0,1,\dots,N$, and let $X_n$ denote
the numerical approximation of the exact solution $X(t_n)$ of (\ref{intro:eq-6}).
According to the expression of the exact solution (\ref{const:eq-2}), we construct the following SEDG scheme,
\begin{equation}\label{const:eq-4}
 \begin{split}
   X_{n+1} & ={} \exp(A h) X_n + h \phi(A h) Q_1 \bar{\nabla} U(X_n,X_{n+1}) \\
                  & + \exp(\frac{A h}{2}) Q_2 \bar{\nabla} V(X_n,X_{n+1}) \Delta W_n,
 \end{split}
\end{equation}
where the scalar function $\phi(z) := (\exp(z)-1)/z$, $X(0) = x_0$, $\Delta W_n = W(t_{n+1})-W(t_n)$, $n=0,1,\dots,N-1$, and the discrete gradient $\bar{\nabla} U$ is a SDG defined in the way of \ref{sdg1}.

Note that, $\Delta W_n$ ($n=0,1,\dots,N-1$) are independent normally distributed random variables with distribution $N(0,h)$, and can be realized by $\xi \sqrt{h}$ with $\xi \sim N(0,1)$. However, the scheme \ref{const:eq-4} is generally implicit so that $\Delta W_n$ may appear in a matrix that need to be invertible, or in the iteration function of a fixed point iteration, so that the infinite variation property of $\Delta W_n$ may cause collapse of the implementation. To overcome the problem, \cite{milstein2002numerical} proposed a replacement of $\xi$ by a suitably truncated bounded random variable $\zeta_h$ as follows
 \begin{equation}\label{const:eq-5}
 \zeta_h = {}
 \begin{cases}
  \xi     & |\xi| \leq C_h, \\
  C_h   & \xi > C_h, \\
  -C_h  & \xi < -C_h,
 \end{cases}
\end{equation}
where $C_h = \sqrt{2 k |\ln{h}|}$ ($k\ge 1$). The error arising from this truncation can be merged into the error of the underlying numerical scheme by choosing $k\ge 2p$ if the numerical scheme is supposed to be of root mean-square convergence order $p$. We use the truncated random variables $\zeta_h\sqrt{h}$ in our numerical scheme, while still  denote them by $\Delta W_n$ for simplicity.

Now we analyze the root mean-square convergence order of the SEDG scheme (\ref{const:eq-4}) for the L-G SDE (\ref{intro:eq-6-1}).
\begin{theorem}\label{const:thm-1}
Suppose the $d$-dimensional stochastic system (\ref{intro:eq-6-1}) satisfies the assumptions (\ref{addeu}) (for $m=1$) and (\ref{addeu2}) for the existence and uniqueness of the solution. In addition, assume that $U,V \in \mathbb{C}^{3}(\mathbb{R}^d)$ with uniformly bounded derivatives
 and $\nabla U, \,\nabla V,$ $Q_2 Hess(V) Q_2 \nabla V$ have bounded second moments along the solution of (\ref{intro:eq-6-1}).
 Then the numerical scheme (\ref{const:eq-4}) is of the first root mean-square convergence order, i.e.,
 \begin{displaymath}
  (E{| X(t_n) - X_n |}^2)^{\frac{1}{2}} = O(h^1), \quad n=1,\dots,N.
 \end{displaymath}
\end{theorem}
\begin{proof}
  Our proof is based on the convergence Theorem 1.1 in \cite{milstein1994numerical}. Let $X(t_n+h)$  be
  the exact evaluation of (\ref{intro:eq-6-1}) at $t_{n+1}$ starting from $X(t_n) = X_n$.
  We can get $X(t_n+h)$ and $X_{n+1}$ by (\ref{const:eq-2}) and (\ref{const:eq-4}), respectively.
  Then we need to calculate the $p_1, p_2$ satisfying
  \begin{equation}\label{sec-2-thm-1:eq-1}
   \begin{split}
    | E(X(t_n+h) - X_{n+1}) | & ={} O(h^{p_1}), \\
    (E{| X(t_n+h) - X_{n+1} |}^2)^{\frac{1}{2}} & ={} O(h^{p_2}).
   \end{split}
  \end{equation}

  It is not difficult to obtain
  \begin{equation}\label{sec-2-thm-1:eq-2}
    X(t_n+h) - X_{n+1}  = P_{1} + P_{2} - P_{3},
  \end{equation}
  where
  \begin{equation*}
   \begin{split}
    P_{1} & ={} \int_{t_n}^{t_{n+1}} \exp(A (t_{n+1}-s)) Q_1 [\nabla U(X(s)) - \bar{\nabla} U(X_n,X_{n+1})] {\rm d}s, \\
    P_{2} & ={} \int_{t_n}^{t_{n+1}} \exp(A (t_{n+1}-s)) Q_2 \nabla V(X(s)) \circ {\rm d}W(s), \\
    P_{3} & ={} \exp(\frac{A h}{2}) Q_2 \bar{\nabla} V(X_n,X_{n+1}) \Delta W_n.
   \end{split}
  \end{equation*}

  To compare $\nabla U(X(s))$ with the corresponding SDG $\bar{\nabla} U(X_n,X_{n+1})$ in $P_{1}$,
  we perform the component expansion of $\bar{\nabla} U(X_n,X_{n+1})$ at $X_n$. Since $U \in \mathbb{C}^{3}(\mathbb{R}^d)$, based on (\ref{const:eq-3}),
  we have
  \begin{equation}\label{sec-2-thm-1:eq-3}
   \bar{\nabla} U^k = \partial_k U + \frac{1}{2} \sum^{d}_{j=1} \partial_{k j} U \, \Delta^j + \frac{1}{2} R_{U^k},\quad k=1,\dots,d,
  \end{equation}
  where the remainder term is
  \begin{equation}\label{sec-2-thm-1:eq-4}
   \begin{split}
    R_{U^k} & ={} \left[ \frac{1}{3} \partial_{k k k} U \, (\Delta^k)^2 + \frac{1}{2} \sum^{d}_{j \not = k} \partial_{k k j} U \, \Delta^k \Delta^j \right. \\
                    & + \frac{1}{2} \sum^{d}_{j \not = k} \partial_{k j j} U \, (\Delta^j)^2 + \sum_{1 \leq m < n \leq k-1} \partial_{k m n} U \, \Delta^m \Delta^n \\
                    & \left. + \sum_{k+1 \leq m < n \leq d} \partial_{k m n} U \, \Delta^m \Delta^n \right] {\Bigg|}_{X_n + \theta \frac{X_{n+1}-X_n}{2}}
   \end{split}
  \end{equation}
  with $0 < \theta < 1$ and $\Delta^j = X_{n+1}^j-X_n^j$, $j=1,\dots, d$.
  Similarly, we can get the component expansion of $\bar{\nabla} V(X_n,X_{n+1})$,
    \begin{equation}\label{sec-2-thm-1:eq-5}
   \bar{\nabla} V^k = \partial_k V + \frac{1}{2} \sum^{d}_{j=1} \partial_{k j} V \, \Delta^j + \frac{1}{2} R_{V^k}, \quad k=1,\dots,d.
  \end{equation}

  Now, for $H = U, V$, we have the expansion of $\bar{\nabla} H(X_n,X_{n+1})$ at $X_n$:
  \begin{equation}\label{sec-2-thm-1:eq-6}
   \bar{\nabla} H = \nabla H + \frac{1}{2} Hess(H) (X_{n+1}-X_n) + \frac{1}{2} R_{H}.
  \end{equation}

  Recall that we denote $f = Q_1 \nabla U$, $g = Q_2 \nabla V$. Then we decompose $P_1$ as follows
  \begin{equation}\label{sec-2-thm-1:eq-7}
  P_{1}={} P_{4} + P_{5},
  \end{equation}
  where
  \begin{equation*}
   \begin{split}
    P_{4} & ={} \int_{t_n}^{t_{n+1}} \exp(A (t_{n+1}-s)) [f(X(s)) - f(X_n)] {\rm d}s, \\
    P_{5} & ={} \int_{t_n}^{t_{n+1}} \exp(A (t_{n+1}-s)) Q_1 [\nabla U(X_n) - \bar{\nabla} U(X_n,X_{n+1})] {\rm d}s.
   \end{split}
  \end{equation*}

Meanwhile,
\begin{equation}\label{sec-2-thm-1:eq-8}
\begin{split}
P_{2} & ={} \int_{t_n}^{t_{n+1}} \exp(A (t_{n+1}-s)) [g(X(s)) - g(X_n)] {\rm d}W(s) \\
		       & + \frac{1}{2} \int_{t_n}^{t_{n+1}} \exp(A ((t_{n+1}-s)) [\frac{\partial g}{\partial x} g(X(s)) - \frac{\partial g}{\partial x} g(X_n)] {\rm d}s \\
		       & + \int_{t_n}^{t_{n+1}} \exp(A ((t_{n+1}-s)) g(X_n) {\rm d}W(s) + \frac{1}{2} h \phi(A h) \frac{\partial g}{\partial x} g(X_n) \\
		       & ={} \int_{t_n}^{t_{n+1}} \exp(A (t_{n+1}-s)) \frac{\partial g}{\partial x} g(X_n) \left( \int_{t_0}^{s} {\rm d}W(t) \right) {\rm d}W(s) \\
		       & + \frac{1}{2} \int_{t_n}^{t_{n+1}} \exp(A (t_{n+1}-s)) [\frac{\partial g}{\partial x} g(X(s)) - \frac{\partial g}{\partial x} g(X_n)] {\rm d}s \\
		      & + \int_{t_n}^{t_{n+1}} \exp(A (t_{n+1}-s)) g(X_n) {\rm d}W(s)
		       + \frac{1}{2} h \phi(A h) \frac{\partial g}{\partial x} g(X_n) + R_{P_{2}},
\end{split}
\end{equation}
where $R_{P_{2}}$ is the higher order remainder term with respect to $h$ resulted from the series expansion of the matrix exponential function inside the integrand, and
\begin{equation}\label{sec-2-thm-1:eq-9}
\begin{split}
 P_{3} & ={} \exp(\frac{A h}{2}) Q_2 \nabla V(X_n) \Delta W_n  \\
		       & + \frac{1}{2} \exp(\frac{A h}{2}) Q_2 Hess(V(X_n)) (X_{n+1}-X_n) \Delta W_n \\
		       & + \frac{1}{2} \exp(\frac{A h}{2}) Q_2 R_{V} \Delta W_n \\
		       & ={} \exp(\frac{A h}{2}) g(X_n) \Delta W_n
		       + \frac{1}{2} \frac{\partial g}{\partial x} g(X_n) (\Delta W_n)^2 +  R_{P_{3}},
\end{split}
\end{equation}
where $R_{P_{3}}$ is the higher order remainder term with respect to $h$ produced by the expansion of the discrete gradient (\ref{sec-2-thm-1:eq-6}). Therefore,
\begin{equation}\label{sec-2-thm-1:eq-10}
    P_{2} - P_{3} = P_{6} + P_{7} + P_{8} + R
\end{equation}
where $R = R_{P_{2}} - R_{P_{3}}$ and
  \begin{equation*}
   \begin{split}
    P_{6} & ={} \left( \int_{t_n}^{t_{n+1}} \exp(A (t_{n+1}-s)) {\rm d}W(s) - \exp(\frac{A h}{2}) \Delta W_n \right) g(X_n), \\
    P_{7} & ={} \frac{1}{2} \int_{t_n}^{t_{n+1}} \exp(A (t_{n+1}-s)) \left(\frac{\partial g}{\partial x} g(X(s)) - \frac{\partial g}{\partial x} g(X_n)\right) {\rm d}s, \\
    P_{8} & ={} \int_{t_n}^{t_{n+1}} \exp(A (t_{n+1}-s)) \frac{\partial g}{\partial x} g(X_n) \left( \int_{t_0}^{s} {\rm d}W(t) \right) {\rm d}W(s) \\
& + \frac{1}{2} (h \phi(A h) - \exp(A h) (\Delta W_n)^2])\frac{\partial g}{\partial x} g(X_n).
   \end{split}
  \end{equation*}

  Based on the triangular inequality and the H\"older inequality, we derive that
  \begin{equation}\label{sec-2-thm-1:eq-11}
   \begin{split}
    | E(X(t_n+h) - X_{n+1}) | & \leq |E P_{4} | + |E P_{5} | + |E P_{6} | + |E P_{7} | + |E P_{8} | + |E R |, \\
    E{| X(t_n+h) - X_{n+1} |}^2 & \leq 6( E |P_{4}|^2 + E |P_{5}|^2  + E |P_{6}|^2  + E |P_{7}|^2  + E |P_{8}|^2 + E |R|^2 ),
   \end{split}
  \end{equation}
  The estimations involve computation of the expectations of multiple It\^o integrals (\cite{platen1982taylor},\cite{tocino2009expectations}).
  According to the properties of the Wiener process (see e.g. \cite{kloeden1992numerical}, \cite{oksendal2013stochastic}),
  \begin{equation*}
   E(\Delta W_n) = 0, \quad E(\Delta W_n)^2 = h, \quad E(\Delta W_n)^3 = 0, \quad E(\Delta W_n)^4 = 3 h^2.
  \end{equation*}
  Under the assumptions on the existence and uniqueness of the solution, and the condition that $U,V \in \mathbb{C}^{3}(\mathbb{R}^d)$ with uniformly bounded derivatives, we can obtain the estimate
  \begin{equation}\label{sec-2-thm-1:eq-12}
   | E(X(t_n+h) - X_{n+1}) | = O(h^2).
  \end{equation}
  It should be noted that the part $R$ is of higher order than other terms.

  According to the H\"older inequality and the Bunyakovsky-Schwarz inequality, we have
  \begin{equation}\label{sec-2-thm-1:eq-13}
   P_{4} = \int_{t_n}^{t_{n+1}} f(X(s)) - f(X_n) {\rm d}s + R_{P_{4}},
  \end{equation}
  with $ R_{P_{4}}$ being the higher order remainder term with respect to $h$ produced by the series expansion of the matrix exponential function.
  \begin{equation}\label{sec-2-thm-1:eq-14}
   \begin{split}
    |P_{4}|^2 & \leq 2 {\left| \int_{t_n}^{t_{n+1}} f(X(s)) - f(X_n) {\rm d}s \right|}^2 + 2 |R_{P_{4}}|^2 \\
    		& \leq  2 h \int_{t_n}^{t_{n+1}} {|f(X(s)) - f(X_n)|}^2 {\rm d}s + 2 |R_{P_{4}}|^2 \\
		& \leq 2 h L^2 \int_{t_n}^{t_{n+1}} {| X(s) - X_n |}^2 {\rm d}s + 2 |R_{P_{4}}|^2, \
   \end{split}
  \end{equation}
  where $L$ is the Lipschitz constant of the function $f$.
  There holds that (see \cite{milstein1994numerical}, p.14)
  \begin{equation*}
   E{|X(t_n + h) - X_n|}^2 \leq K (1+E{|X_n|}^2) h,
  \end{equation*}
  which, together with \ref{sec-2-thm-1:eq-14}, implies that $E|P_4|^2= O(h^3)$. Similarly, for $P_{5}$ and $P_{7}$ we have  $E|P_5|^2= O(h^3)$, $E|P_7|^2= O(h^3) $. Using the fact that
  \begin{equation}\label{sec-2-thm-1:eq-15}
   \int_{t_n}^{t_{n+1}} \left( \int_{t_n}^{s} {\rm d}W(t) \right) {\rm d}W(s) = \frac{1}{2} ( (\Delta W_n)^2 - h ),
  \end{equation}
  we can also derive $E|P_{6}|^2=O(h^3)$, $E|P_{8}|^2=O(h^3)$ .
  Then, we obtain
  \begin{equation}\label{sec-2-thm-1:eq-16}
   E{|X(t_n+h) - X_{n+1} |}^2 = O(h^3).
  \end{equation}
  Finally, (\ref{sec-2-thm-1:eq-12}) and (\ref{sec-2-thm-1:eq-16}) imply $p_1 = 2, p_2 = \frac{3}{2}$.
  By the Theorem 1.1 in \cite{milstein1994numerical},
  the root mean-square order of the SEDG scheme (\ref{const:eq-4}) for the L-G SDE (\ref{intro:eq-6}) is $p = p_2 - \frac{1}{2} = 1$.
\end{proof}

% supplement
%------------------------------------------------
Now we generalize our method to the L-G SDEs with multiple noises (\ref{intro:eq-6}). In this case, the SEDG scheme reads
\begin{equation}\label{const:eq-7}
 \begin{split}
   X_{n+1} & ={} \exp(A h) X_n + h \phi(A h) Q_1 \bar{\nabla} U(X_n,X_{n+1}) \\
                  & + \sum_{r=1}^{m} \exp(\frac{A h}{2}) Q_{2 , r} \bar{\nabla} V_{r}(X_n,X_{n+1}) \Delta W_{r , n}
 \end{split}
\end{equation}
where $X(0) = x_0$, and $\Delta W_{r , n} \sim \zeta_{r,h}\sqrt{h}$ are simulations of $W_{r}(t_{n+1})-W_{r}(t_n)$, $n=0,1,\dots,N-1$, $r = 1,\dots,m$.
Correspondingly, we present the following convergence theorem without proof, since it is similar to that of Theorem \ref{const:thm-1}.

\begin{theorem}\label{const:thm-2}
 Suppose the $d$-dimensional stochastic system (\ref{intro:eq-6}) satisfies the assumptions (\ref{addeu}) and (\ref{addeu2}) for
 the existence and uniqueness of the solution. In addition, assume that $U,\,V_{r} \in \mathbb{C}^{3}(\mathbb{R}^d)$ with uniformly bounded derivatives, $r = 1,\dots,m$, and $\nabla U, \,\nabla V_{r},$ $Q_{2,r} Hess(V_{r}) Q_{2,r} \nabla V_{r}$, $r = 1,\dots,m$ have bounded second moments along the solution of (\ref{intro:eq-6}),
 then the numerical scheme (\ref{const:eq-7}) is of first root mean-square convergence order, i.e.,
 \begin{displaymath}
  (E{| X(t_n) - X_n |}^2)^{\frac{1}{2}} = O(h^1), \quad n=1,\dots,N.
 \end{displaymath}
\end{theorem}

%------------------------------------------------
%Section 3
%------------------------------------------------------

\section{SEDG methods for L-G SDEs with special structures}
\label{sec:struct}
In this section,  we investigate the accuracy and the structure-preserving properties of the SEDG schemes applied to certain L-G SDEs with special structures.
% subsection 1
%------------------------------------------------
\subsection{SEDG methods for a class of stochastic highly oscillatory systems}
\label{sec:analy}
Consider the nonlinear stochastic oscillators  with high frequency (\cite{cohen2012numerical}, \cite{vilmart2014weak})
\begin{equation}\label{intro:addosc}
\ddot{X_t} + \omega^2 X_t ={} f(X_t) + g(X_t)\circ \dot{W_t},
\end{equation}
where $\omega$ is a large positive constant, and $f,\,g$ are smooth real functions.

We first present a one-step error analysis for the SEDG method applied to (\ref{intro:addosc}). In order to facilitate the analysis, we rewrite (\ref{intro:addosc}) into
\begin{equation}\label{analy:eq-2}
{\rm d} \begin{pmatrix}
      x^1 \\
      x^2
      \end{pmatrix} ={} \begin{pmatrix}
      0 & -\omega^2 \\
      1 & 0
      \end{pmatrix} \begin{pmatrix}
      x^1 \\
      x^2
      \end{pmatrix} {\rm d}t + \begin{pmatrix}
      f(x^2) \\
      0
      \end{pmatrix} {\rm d}t + \begin{pmatrix}
      g(x^2) \\
      0
      \end{pmatrix} \circ {\rm d}W(t).
\end{equation}
Besides, we assume that $f = \nabla U$, $g = \nabla V$.

Due to the high frequency $\omega$ and the expected usage of reasonable time step sizes, we assume $h \ge O(\frac{1}{\omega}) $.
Our aim is to estimate the root mean-square error. For convenience, we denote $\bar{\nabla} U(x^2_n,x^2_{n+1})$ and $\bar{\nabla} V(x^2_n,x^2_{n+1})$ by $\bar{\nabla} U$ and $\bar{\nabla} V$, respectively.

By applying (\ref{const:eq-4}) to (\ref{analy:eq-2}), we can obtain
\begin{equation}\label{analy:eq-3}
\begin{split}
x^1 (t_{n+1}) - x^1_{n+1} & ={} \int_{t_n}^{t_{n+1}} \cos(\omega (t_{n+1}-s)) [\nabla U(x^2(s)) - \bar{\nabla} U] {\rm d}s \\
				    & + \frac{1}{2} \int_{t_n}^{t_{n+1}} \cos(\omega (t_{n+1}-s)) Hess(V) \nabla V(x^2(s)) {\rm d}s \\
				    & + \int_{t_n}^{t_{n+1}} \cos(\omega (t_{n+1}-s)) [\nabla V(x^2(s)) - \nabla V(x^2_n)] {\rm d}W(s) \\
				    & + \nabla V(x^2_n) \int_{t_n}^{t_{n+1}} \cos(\omega (t_{n+1}-s)) {\rm d}W(s) \\
				    & - \bar{\nabla} V \cos(\frac{\omega h}{2}) \Delta W_n.
\end{split}
\end{equation}

Under the impact of high oscillation, we note that $$\bar{\nabla} D = \nabla D(x^2_n) + \frac{1}{2} Hess(D) (\cos(\omega h) - 1) x^2_n + O({\omega}^{-1}),\quad \mbox{for} \,\,D = U, V.$$
Thus, we are only concerned with the low order term of the right-hand-side of (\ref{analy:eq-3})
\begin{equation*}
P_{x^1, low} :={} \nabla V(x^2_n) \int_{t_n}^{t_{n+1}} \cos(\omega (t_{n+1}-s)) {\rm d}W(s) - \bar{\nabla} V \cos(\frac{\omega h}{2}) \Delta W_n.
\end{equation*}

By the H\"older inequality and the It\^o isometry, we get $E{| P_{x^1, low}|}^2 = O(h)$. Similarly, replacing $\cos(\omega h)$ by ${\omega}^{-1} \sin(\omega h)$ to calculate $x^2 (t_{n+1}) - x^2_{n+1}$, we get $E{| P_{x^2, low}|}^2 = O({\omega}^{-2} h)$. Thus
\begin{equation}\label{analy:eq-5}
\begin{split}
( E{| x^1 (t_{n+1}) - x^1_{n+1} |}^2 )^{\frac{1}{2}} & ={} O( h^{\frac{1}{2}}), \\
( E{| x^2 (t_{n+1}) - x^2_{n+1} |}^2 )^{\frac{1}{2}} & ={} O( \omega^{-1} h^{\frac{1}{2}} ).
\end{split}
\end{equation}

Note that, under our assumptions, the oscillator (\ref{analy:eq-2}) is a stochastic Hamiltonian system. So we compare our SEDG scheme with the symplectic Euler-Maruyama (SEM) scheme, one of the standard methods for solving stochastic Hamiltonian systems. To distinguish the symbols for the SEM scheme from those of the SEDG, we use symbol $y$ instead of $x$ for the SEM method. Analogously, we have
\begin{equation}\label{analy:eq-6}
\begin{split}
y^1 (t_{n+1}) - y^1_{n+1} & ={} - \omega^{2} \int_{t_n}^{t_{n+1}} [y^2(s) - y^2_n] {\rm d}s \\
				    & + \int_{t_n}^{t_{n+1}} ( \nabla U + \frac{1}{2} Hess(V) \nabla V ) {\Big|}^{y^2(s)}_{y^2_n} {\rm d}s \\
				    & + \int_{t_n}^{t_{n+1}} [\nabla V(x^2(s)) - \nabla V(x^2_n)] {\rm d}W(s).
\end{split}
\end{equation}
Again, we consider the low order term of the right-hand-side of (\ref{analy:eq-6})
\begin{equation*}
P_{y^1, low} :={} \omega^{2} \int_{t_n}^{t_{n+1}} [y^2(s) - y^2_n] {\rm d}s.
\end{equation*}
By the Bunyakovsky-Schwarz inequality, we get $E{| P_{y^1, low}|}^2 = O(\omega^4 h^4)$. Therefore,
\begin{equation}\label{analy:eq-7}
( E{| y^1 (t_{n+1}) - y^1_{n+1} |}^2 )^{\frac{1}{2}} ={} O( \omega^2 h^2 ).
\end{equation}
Then we calculate $y^2 (t_{n+1}) - y^2_{n+1} = \int_{t_n}^{t_{n+1}} [y^1(s) - y^1_{n+1}] {\rm d}s$. By (\ref{analy:eq-7}), it is not difficult to obtain
\begin{equation}\label{analy:eq-8}
( E{| y^2 (t_{n+1}) - y^2_{n+1} |}^2 )^{\frac{1}{2}} ={} O( \omega^2 h^3 ).
\end{equation}

We can see from (\ref{analy:eq-5}), (\ref{analy:eq-7}) and (\ref{analy:eq-8}) that, with the increase of the frequency parameter $\omega$, the SEDG method becomes more accurate than the SEM method. The main reason underlying the fact is the exponential structure of the SEDG scheme which cancels the lower order terms in the error with respect to $\omega$, while other non-exponential integrators can not cancel out such terms in the error.

As an illustration we apply the SEDG scheme to a concrete highly oscillatory 2-dimensional stochastic Hamiltonian system
\begin{equation}\label{exam-4:eq-1}
{\rm d} \begin{pmatrix}
      x^1 \\
      x^2
      \end{pmatrix} ={} \begin{pmatrix}
      0 & -\omega^2 \\
      1 & 0
      \end{pmatrix} \begin{pmatrix}
      x^1 \\
      x^2
      \end{pmatrix} {\rm d}t + \begin{pmatrix}
      \sigma \\
      0
      \end{pmatrix} \circ {\rm d}W(t), \quad \begin{pmatrix}
      x^1(0) \\
      x^2(0)
      \end{pmatrix} =\begin{pmatrix}
      x^1_0\\
      x^2_0
      \end{pmatrix} ,
      \end{equation}
with $\omega > 0$ and $\sigma \not = 0$. It has the Hamiltonian functions $H_1=\frac{1}{2}((x^1)^2+\omega^2(x^2)^2)$ and $H_2=-\sigma x^2$. For large $\omega$, the oscillator is a stiff SDE with high frequency. A simple derivation or direct referring to the literature such as \cite{milstein2002symplectic} yields the symplecticity of the system
\begin{equation}\label{addsymp}
dx^1(t)\wedge dx^2(t)=dx^1_0\wedge dx^2_0,\quad \forall t\ge 0.
\end{equation}
Moreover, there is a linear growth property with respect to $t$ for the expectation of the function $H_1$ (see e.g. \cite{cohen2012numerical}), i.e.

\begin{equation}\label{exam-4-thm-1}
E \left[ \frac{1}{2} ( {x^1}(t)^2 + \omega^2 {x^2}(t)^2 ) \right] = \frac{1}{2} ((x^1_0)^2 + \omega^2 (x^2_0)^2 ) + \frac{\sigma^2}{2} t,\quad \forall t\ge0.
\end{equation}

The oscillator (\ref{exam-4:eq-1}) can be written in the form of the general L-G SDE (\ref{intro:eq-6-1}) if we take $A = \begin{pmatrix}
      0 & -\omega^2 \\
      1 & 0
      \end{pmatrix}$, $Q_1 = 0$, $Q_2 = I$ and $V=\sigma x^1$ in (\ref{intro:eq-6-1}). Then the SEDG scheme (\ref{const:eq-4}) applied to the oscillator (\ref{exam-4:eq-1}) reads
\begin{equation}\label{exam-4:eq-2}
\begin{split}
 x^1_{n+1} ={} & \cos(h \omega) x^1_{n} - \omega \sin(h \omega) x^2_{n} + \sigma \cos(\frac{h \omega}{2}) \Delta W_n, \\
 x^2_{n+1} ={} & \omega^{-1} \sin(h \omega) x^1_{n} + \cos(h \omega) x^2_{n} + \sigma \omega^{-1} \sin(\frac{h \omega}{2}) \Delta W_n.
\end{split}
\end{equation}
\begin{theorem}\label{add:thm} The SEDG scheme (\ref{exam-4:eq-2}) for the highly oscillatory SHS (\ref{exam-4:eq-1}) preserves the symplecticity (\ref{addsymp}), as well as the linear growth property (\ref{exam-4-thm-1}) of the system (\ref{exam-4:eq-1}).
\end{theorem}
\begin{proof}
Obviously, according to (\ref{exam-4:eq-2})
\begin{equation*}
dx^1_{n+1}\wedge dx^2_{n+1}=(\cos^2 (h\omega)+\sin^2 (h\omega)) dx^1_n\wedge dx^2_n=dx^1_n\wedge dx^2_n, \quad \forall n\ge 0.
\end{equation*}
Thus \ref{exam-4:eq-2} preserves the symplecticity of the original system. On the other hand, a straightforward calculation gives
\begin{equation*}
\begin{split}
E((x^1_{n+1})^2+\omega^2 (x^2_{n+1})^2)&=E((x^1_n)^2+\omega^2(x^2_n)^2)+\sigma^2 h\\
                                                                    &=(x^1_0)^2+\omega^2 (x^2_0)^2+(n+1)\sigma^2 h \\
                                                                     &=(x^1_0)^2+\omega^2 (x^2_0)^2+\sigma t_{n+1},\quad \forall n\ge 0.
\end{split}
\end{equation*}
\end{proof}
Next we show that the SEM does not preserve the linear growth property.  Note the SEM for (\ref{exam-4:eq-1})  reads
\begin{equation}\label{addsem}
\begin{split}
x^1_{n+1}&=x^1_n-\omega^2 h x^2_n+\sigma \Delta W_n,\\
x^2_{n+1}&=x^2_n+h x^1_{n+1}.
\end{split}
\end{equation}
Then it is easy to see that
\begin{equation*}
E((x^1_{n+1})^2+\omega^2 (x^2_{n+1})^2)=E((x^1_{n})^2+\omega^2 (x^2_{n})^2)+\sigma^2 h+R_n,
\end{equation*}
where
$$R_n=\omega^2 h^2 [E((x^1_n)^2-\omega^2 (x^2_n)^2+\omega^4h^2(x^2_n)^2-2\omega^2hx^1_nx^2_n)+\sigma^2h].$$
Obviously,  the symplectic Euler-Maruyama method can not preserve the linear growth property.
%------------------------------------------------

% subsection 2
%------------------------------------------------
\subsection{A class of energy-preserving stochastic Poisson systems}
Consider the stochastic Poisson system (\cite{cohen1})
\begin{equation}\label{intro:poisson}
 {\rm d}X(t) = Q (M X(t) + \nabla U(X(t))) ( {\rm d}t + \sigma \circ {\rm d}W(t) ), \quad X(t_0) = x_0,
\end{equation}
where $Q$ is a skew-symmetric and nonsingular matrix, $M$ is a symmetric and nonsingular matrix, and $U \in \mathbb{C}^{3}(\mathbb{R}^d)$. It
has the energy $\mathcal{H}(X)=\frac{1}{2} X^{\mathrm{T}} M X + U(X) $ (\cite{cohen1}).
\begin{theorem}\label{struct:thm-1}
 $\mathcal{H}(X) = \frac{1}{2} X^{\mathrm{T}} M X + U(X)$  is a first integral of the stochastic system (\ref{intro:poisson}).
\end{theorem}
\begin{proof}
  Since the coefficient matrix $Q$ is skew symmetric, and $M$ is symmetric, then according to (\ref{intro:poisson}) we have
  \begin{equation*}
  {\rm d}\mathcal{H}(X) ={} (M X + \nabla U(X))^{\mathrm{T}} {\rm d} X = 0,
  \end{equation*}
 which means that the quantity $\mathcal{H}(X)$ is invariant.
\end{proof}

 Note that there is a linear part in the diffusion coefficient of the L-G SDE (\ref{intro:poisson}). To make full use of this structural feature, we modify the change of variables for the general L-G SDE (\ref{intro:eq-6-1}) to be $Z(t) = \exp(- A (t+\sigma W(t)) ) X(t)$,
 where $A = Q M$. Then we have
\begin{equation*}
 \begin{split}
  {\rm d}Z(t) ={} & - A \exp(- A (t+\sigma W(t)) ) X(t) ( {\rm d}t + \sigma \circ {\rm d}W(t) ) \\
  	          + & \exp(- A (t+\sigma W(t)) ){\rm d}X(t) \\
                    ={} & \exp(- A (t+\sigma W(t)) ) Q \nabla U(X(t)) ( {\rm d}t + \sigma \circ {\rm d}W(t) )
 \end{split}
\end{equation*}

Defining $E_A (t_0,t) := \exp(A [t - t_0 + \sigma ( W(t) - W(t_0) ) ])$, we reformulate (\ref{intro:poisson}) in integral form on $[t_0, t_0+h]$ for any $h\ge 0$,
\begin{equation}\label{struct:eq-2}
X(t_0+h) ={}  E_A (t_0,t_0+h) x_0+ \int_{t_0}^{t_0+h} E_A (s,t_0+h) Q \nabla U(X(s)) ( {\rm d}s + \sigma \circ {\rm d}W(s) ).
\end{equation}

Using the symmetric discrete gradient (\ref{sdg1}) for $\nabla U$,  and abbreviating $E_A(t_n,t_n+h)$ with $E_A(t_n)$, we construct the scheme
\begin{equation}\label{struct:eq-3}
X_{n+1} ={} E_A (t_n) X_n + (E_A (t_n) - I) M^{-1} \bar{\nabla} U(X_n,X_{n+1}),
\end{equation}
with $X_0=x_0$.
\begin{lemma}\label{struct:lem-1}
 For any symmetric matrix $M$ and scalar $h \geq 0$, $n\ge 0$, let
 \begin{displaymath}
   B_n = E_A (t_n)^{\mathrm{T}} M E_A (t_n) - M,
 \end{displaymath}
where $A=QM$, $ E_A (t_n) = \exp(A (h + \sigma \Delta W_n ))$, and $\sigma \not = 0$.
If $Q$ is a skew symmetric matrix, then $B_n = 0$.
\end{lemma}
\begin{proof}
Note that $B_n$ is symmetirc. To show $B_n=0$ it suffices to show that $B_n$ is also skew symmetric. Consider the SDE
\begin{equation}\label{sec-2-lem-1:eq-1}
  {\rm d}Y(t) = A Y(t) ( {\rm d}t + \sigma \circ {\rm d}W(t) ), \quad Y(t_0) = y_0.
\end{equation}
Comparing the above with (\ref{intro:poisson}), then it is not difficult to see that $\frac{1}{2} Y^{\mathrm{T}} M Y$ is invariant for (\ref{sec-2-lem-1:eq-1}).
Since $Y(t_n+h) = E_A (t_n)  Y(t_n)$, we have
\begin{equation}\label{sec-2-lem-1:eq-2}
  \frac{1}{2} Y(t_n)^{\mathrm{T}} E_A (t_n)^{\mathrm{T}} M E_A (t_n) Y(t_n)= \frac{1}{2} Y(t_n)^{\mathrm{T}} M Y(t_n),\quad \forall \, n\ge 0.
\end{equation}
Then we know that  $B_n$ is skew symmetric.
\end{proof}
If $\sigma = 0$, we get the deterministic version of Lemma \ref{struct:lem-1} (also see \cite{li2016exponential}).

\begin{theorem}\label{struct:thm-2}
The SEDG scheme (\ref{struct:eq-3}) for the stochastic system (\ref{intro:poisson}) preserves the energy $\mathcal{H}(X) = \frac{1}{2} X^{\mathrm{T}} M X + U(X)$ exactly, i.e.
\begin{equation*}
\mathcal{H}(X_{n+1}) = \mathcal{H}(X_{n}), \quad n=0,1,\dots,N-1.
\end{equation*}
\end{theorem}
\begin{proof}
  For convenience, we denote $\bar{\nabla} U(X_n,X_{n+1})$ and $E_A (t_n) ={} \exp(A (h + \sigma \Delta W_n ) $ by $\bar{\nabla} U$ and $E_A $, respectively.
  Firstly, we calculate $\frac{1}{2} X_{n+1}^{\mathrm{T}} M X_{n+1}$ according to (\ref{struct:eq-3}).
 \begin{equation}\label{struct-thm-2:eq-1}
  \begin{split}
   \frac{1}{2} X_{n+1}^{\mathrm{T}} M X_{n+1} ={} & \frac{1}{2} \left( E_A X_n + (E_A - I) M^{-1} \bar{\nabla} U \right)^{\mathrm{T}} \\
   	& M \left( E_A X_n + (E_A - I) M^{-1} \bar{\nabla} U \right) \\
	={} & \frac{1}{2} X_{n}^{\mathrm{T}} E_A^{\mathrm{T}} M E_A X_{n} + X_{n}^{\mathrm{T}} E_A^{\mathrm{T}} M (E_A - I) (M^{-1} \bar{\nabla} U) \\
	+ & \frac{1}{2} (M^{-1} \bar{\nabla} U)^{\mathrm{T}} (E_A - I)^{\mathrm{T}} M (E_A - I) (M^{-1} \bar{\nabla} U).
  \end{split}
 \end{equation}

  Secondly, we calculate $U(X_{n+1}) - U(X_n)$ basing on Definition \ref{const:de-2} and (\ref{struct:eq-3}).
  \begin{equation}\label{struct-thm-2:eq-2}
   \begin{split}
    U(X_{n+1}) - U(X_n) ={} & \left( X_{n+1} - X_{n} \right)^{\mathrm{T}} \bar{\nabla} U \\
	={} &  X_{n}^{\mathrm{T}} ( E_A - I )^{\mathrm{T}} \bar{\nabla} U + (M^{-1} \bar{\nabla} U)^{\mathrm{T}} (E_A - I)^{\mathrm{T}} \bar{\nabla} U.
   \end{split}
  \end{equation}

  Then, combining (\ref{struct-thm-2:eq-1}) and (\ref{struct-thm-2:eq-2}), we get
  \begin{equation}\label{struct-thm-2:eq-3}
   \begin{split}
    \mathcal{H}(X_{n+1}) - \mathcal{H}(X_n) ={} & \frac{1}{2} X_{n}^{\mathrm{T}} B_n X_{n} + X_{n}^{\mathrm{T}} B_n (M^{-1} \bar{\nabla} U) \\
	+ & \frac{1}{2} (M^{-1} \bar{\nabla} U)^{\mathrm{T}} (B_n + E_A^{\mathrm{T}} M - M E_A) (M^{-1} \bar{\nabla} U) \\
	={} &0,
   \end{split}
  \end{equation}
  where $B_n = E_A^{\mathrm{T}} M E_A - M$.
  The last step is according to Lemma \ref{struct:lem-1} and this completes the proof.
\end{proof}

Note that the scheme (\ref{struct:eq-3}) for the system (\ref{intro:poisson}) is constructed in a slightly different way from that of the general SEDG scheme (\ref{const:eq-4}) for the general L-G SDE (\ref{intro:eq-6-1}). As a result, a slight modification of the proof of Theorem \ref{const:thm-1} is needed to prove the error estiamte of (\ref{struct:eq-3}).  We state the result in the following theorem, and put its proof in the Appendix.

\begin{theorem}\label{struct:thm-3}
 Denote $f=g_1=\nabla U$, suppose the $d$-dimensional stochastic system (\ref{intro:poisson}) (let $A = Q M$)
 satisfies the assumptions for the existence and uniqueness of the solution (\ref{addeu}) (for $m=1$) and (\ref{addeu2}).
 In addition, assume that $U \in \mathbb{C}^{3}(\mathbb{R}^d)$ with uniformly bounded derivatives
 and $\nabla U$, as well as $Q Hess(U) (A X + Q \nabla U)$ have bounded second moments along the solution of (\ref{intro:poisson}).
 Then the numerical scheme (\ref{struct:eq-3}) is of root mean-square convergence order 1, i.e.,
 \begin{displaymath}
  (E{| X(t_n) - X_n |}^2)^{\frac{1}{2}} = O(h^1), \quad n=1,\dots,N.
 \end{displaymath}
\end{theorem}

%------------------------------------------------

% subsection 3
%------------------------------------------------
\subsection{The stochastic Langevin-type equations}
Consider the $2\bar{d}$-dimensional stochastic Langevin-type equation
\begin{equation}\label{lang:eq-1}
\begin{split}
 {\rm d}P ={} & f_1 (Q) {\rm d}t - \nu P {\rm d}t + \sigma \circ {\rm d}W(t), \quad P(t_0) = p_0, \\
 {\rm d}Q ={} & M^{-1} P {\rm d}t, \quad Q(t_0) = q_0,
\end{split}
\end{equation}
where $\nu \ge 0$ is a parameter, $\sigma\in\mathbb{R}^{\bar{d}}$ is a constant vector, and $M$ is a $\mathbb{R}^{\bar{d}\times \bar{d}}$ positive definite matrix.
Assume that there exists a scalar function $U_0 (Q)$ such that
\begin{equation}\label{lang:eq-2}
  f_1^i (Q) ={} - \frac{\partial U_0}{\partial Q^i}, \quad i = 1, \dots, \bar{d}.
\end{equation}

\begin{lemma}[see \cite{hong2017high}]\label{lang:de-1}
The symplectic 2-form dissipates exponentially along the phase flow of the system (\ref{lang:eq-1}), i.e.,
\begin{equation}\label{confsym1}
 {\rm d}P(t) \wedge {\rm d}Q(t) ={} \exp(- \nu t) {\rm d}p_0 \wedge {\rm d}q_0, \quad \forall t \ge 0.
\end{equation}
\end{lemma}

Because of the above property, the system (\ref{lang:eq-1})  is said to preserve conformal symplectic structure. In fact, when $\nu = 0$,  (\ref{lang:eq-1}) becomes a stochastic Hamiltonian system which preserves symplectic structure, i.e., $dP(t)\wedge dQ(t)=dp_0\wedge dq_0,\,\,\forall t\ge 0$. In the following we assume $\nu>0$.

Let $X=\begin{pmatrix}
      P \\
      Q
      \end{pmatrix}$, $A = \begin{pmatrix}
      - \nu I_{\bar{d}} & 0 \\
      M^{-1} & 0
      \end{pmatrix}$, $Q_1=\begin{pmatrix}0 & -I_{\bar{d}} \\I _{\bar{d}}& 0\end{pmatrix}$, $U=U_0$, $Q_2=I_{2\bar{d}}$, $V=\sigma P$, the equation (\ref{lang:eq-1}) can be written in the form of the general L-G SDE (\ref{intro:eq-6-1}). Applying the SEDG method (\ref{const:eq-4})  to this system, we obtain the following scheme
\begin{equation}\label{lang:eq-5}
\begin{split}
 P_{n+1} ={} & \exp(- \nu h) P_{n} - \frac{1-\exp(- \nu h)}{\nu} \bar{\nabla} U_0 (Q_n,Q_{n+1}) \\
 	     + & \exp(- \frac{\nu h}{2}) \sigma \Delta W_n, \\
 Q_{n+1} ={} & Q_{n} + \frac{1-\exp(- \nu h)}{\nu} M^{-1} P_{n} \\
 	      + & \frac{1}{\nu} \left( \frac{1-\exp(- \nu h)}{\nu} - h \right) M^{-1} \bar{\nabla} U_0 (Q_n,Q_{n+1}) \\
 	      + & \frac{1-\exp(- \frac{\nu h}{2})}{\nu} M^{-1} \sigma \Delta W_n.
\end{split}
\end{equation}

Next we study to what degree of accuracy can the scheme (\ref{lang:eq-5}) preserve the conformal symplecticity of the system (\ref{lang:eq-1}) characterized by Lemma \ref{lang:de-1}.
For convenience and referring to \cite{Hairer2006Geometric}, we check the conformal symplecticity  using its equivalent description.
\begin{theorem}\label{lang:thm-1}
Suppose for any $Q$, $\hat{Q}$, $\partial_1\bar{\nabla} U_0 (Q,\hat{Q})$ and $\partial_2\bar{\nabla} U_0 (Q,\hat{Q})$ are symmetric and bounded. The numerical scheme (\ref{lang:eq-5}) for the stochastic Langevin-type equation (\ref{lang:eq-1})
nearly preserves the conformal symplectic structure within error of root mean-square order 2, i.e.,
\begin{equation}\label{lang:eq-6}
\left( \frac{\partial{(P_{n+1},Q_{n+1})}}{\partial{(P_n,Q_n)}} \right)^{\mathrm{T}} J \left( \frac{\partial{(P_{n+1},Q_{n+1})}}{\partial{(P_n,Q_n)}} \right) ={} \exp(- \nu h) J + R^{CS},
\end{equation}
where for the remainder matrix $R^{CS}$, $(E|R^{CS}|^2)^{\frac{1}{2}}\le O(h^2)$ .
\end{theorem}
\begin{proof}
For convenience, denote $\bar{\nabla} U_0 (Q_n,Q_{n+1})$ and $\partial_i \bar{\nabla} U_0 (Q_n,Q_{n+1})$ ($i=1,2$) by $\bar{\nabla} U_0$ and $\partial_i \bar{\nabla} U_0 $, respectively, and define $$\bar{\nu} = \frac{1-\exp(- \nu h)}{\nu}.$$ In fact, based on the definition of a SDG, if one of  $\partial_1\bar{\nabla}U_0$ and $\partial_2\bar{\nabla}U_0$ is symmetric then the other is also symmetric. By direct calculations and the symmetry of the positive definite matrix $M$, we have
\begin{equation}\label{addpar}
\begin{split}
\frac{\partial P_{n+1}}{\partial P_n}^T\frac{\partial Q_{n+1}}{\partial P_n}-\frac{\partial Q_{n+1}}{\partial P_n}^T\frac{\partial P_{n+1}}{\partial P_n}&=0,\\
\frac{\partial P_{n+1}}{\partial Q_n}^T\frac{\partial Q_{n+1}}{\partial Q_n}-\frac{\partial Q_{n+1}}{\partial Q_n}^T\frac{\partial P_{n+1}}{\partial Q_n}&=O(h^3).
\end{split}
\end{equation}
It should be noted that the derivation of (\ref{addpar}) is based on the facts that
$$\frac{\partial{Q_{n+1}}}{\partial{P_n}} = \bar{\nu} \left( M - \frac{\bar{\nu} - h}{\nu}\partial_2\bar{\nabla} U_0\right)^{-1}\quad \mbox{is symmetric},$$
\begin{equation}\label{addqq}
\frac{\partial{Q_{n+1}}}{\partial{Q_n}} = \left( I - \frac{\bar{\nu} - h}{\nu} M^{-1} \partial_2\bar{\nabla} U_0 \right)^{-1}
	 \left( I + \frac{\bar{\nu} - h}{\nu} M^{-1} \partial_1\bar{\nabla} U_0 \right),
\end{equation}
and the Sherman-Morrison-Woodbury formula
\begin{equation*}
(A - U C V)^{-1} = A^{-1} - A^{-1} U (C^{-1} + V A^{-1} U)^{-1} V A^{-1},
\end{equation*}
where $A$, $U$, $C$ and $V$ are matrices of conformable sizes. Moreover, since $\frac{\bar{\nu}-h}{\nu}=O(h^2)$, the invertibility of $I - \frac{\bar{\nu} - h}{\nu} M^{-1} \partial_2\bar{\nabla} U_0 $ and therefore also $ M - \frac{\bar{\nu} - h}{\nu}\partial_2\bar{\nabla} U_0$ can be guaranteed by sufficiently small $h$, the boundedness of $\partial_2\bar{\nabla}U_0$, and the non-singularity of $M$.

Consequently, (\ref{lang:eq-6}) holds if and only if
\begin{equation}\label{lang:eq-7}
{\frac{\partial{P_{n+1}}}{\partial{P_n}}}^{\mathrm{T}} \frac{\partial{Q_{n+1}}}{\partial{Q_n}} - {\frac{\partial{Q_{n+1}}}{\partial{P_n}}}^{\mathrm{T}} \frac{\partial{P_{n+1}}}{\partial{Q_n}} ={} \exp(- \nu h) I + \bar{R}^{CS},
\end{equation}
where the remainder matrix $\bar{R}^{CS}$ satisfies $(E|\bar{R}^{CS}|^2)^{\frac{1}{2}}\le O(h^2)$.
For brevity, we denote the left part of (\ref{lang:eq-7}) by $\bar{I}$. Then by (\ref{lang:eq-5}), we have
\begin{equation}\label{addfinal}
\bar{I} ={} \exp(- \nu h) \frac{\partial{Q_{n+1}}}{\partial{Q_n}}
	 + \bar{\nu}^2 \left( I - \frac{\bar{\nu} - h}{\nu}M^{-1} \partial_2\bar{\nabla} U_0 \right)^{-1} M^{-1}\partial_1\bar{\nabla} U_0.
\end{equation}
Substituting (\ref{addqq}) into (\ref{addfinal}), and performing Taylor expansion of the matrix function $$\left( I - \frac{\bar{\nu} - h}{\nu}M^{-1} \partial_2\bar{\nabla} U_0 \right)^{-1}$$ according to the rule of expanding  $(1-x)^{-1}$,  we conclude that (\ref{lang:eq-7}) holds.
\end{proof}

%------------------------------------------------

% section 4
%------------------------------------------------
% section 5
\section{Numerical experiments}
\label{sec:numer}
In this section, we verify the behavior of the SEDG schemes via numerical experiments on different models, in particular we exam their convergence and structure-preserving properties, and compare them with other existing numerical methods.

Throughout the numerical experiments, we simulate the reference exact solutions of the model systems either by plotting the analytical solution if it exists, or by numerical realizations on tiny time step sizes, if the system has no analytical solution.

%------------------------------------------------
% subsection 1
\subsection{A stochastic Hamiltonian oscillator with high frequency}
For the highly oscillatory stochastic Hamiltonian system (\ref{exam-4:eq-1})
\begin{equation*}
{\rm d} \begin{pmatrix}
      x^1 \\
      x^2
      \end{pmatrix} ={} \begin{pmatrix}
      0 & -\omega^2 \\
      1 & 0
      \end{pmatrix} \begin{pmatrix}
      x^1 \\
      x^2
      \end{pmatrix} {\rm d}t + \begin{pmatrix}
      \sigma \\
      0
      \end{pmatrix} \circ {\rm d}W(t), \quad \begin{pmatrix}
      x^1(0) \\
      x^2(0)
      \end{pmatrix} =\begin{pmatrix}
      x^1_0\\
      x^2_0
      \end{pmatrix} ,
      \end{equation*}
we implement the SEDG scheme (\ref{exam-4:eq-2}) and the symplectic Euler-Maruyama (SEM) method (\ref{addsem}), and illustrate their numerical performance.
  \begin{figure}[htbp]
  \centering
  \label{exam-4:fig-1-a}\includegraphics[width=0.45\textwidth]{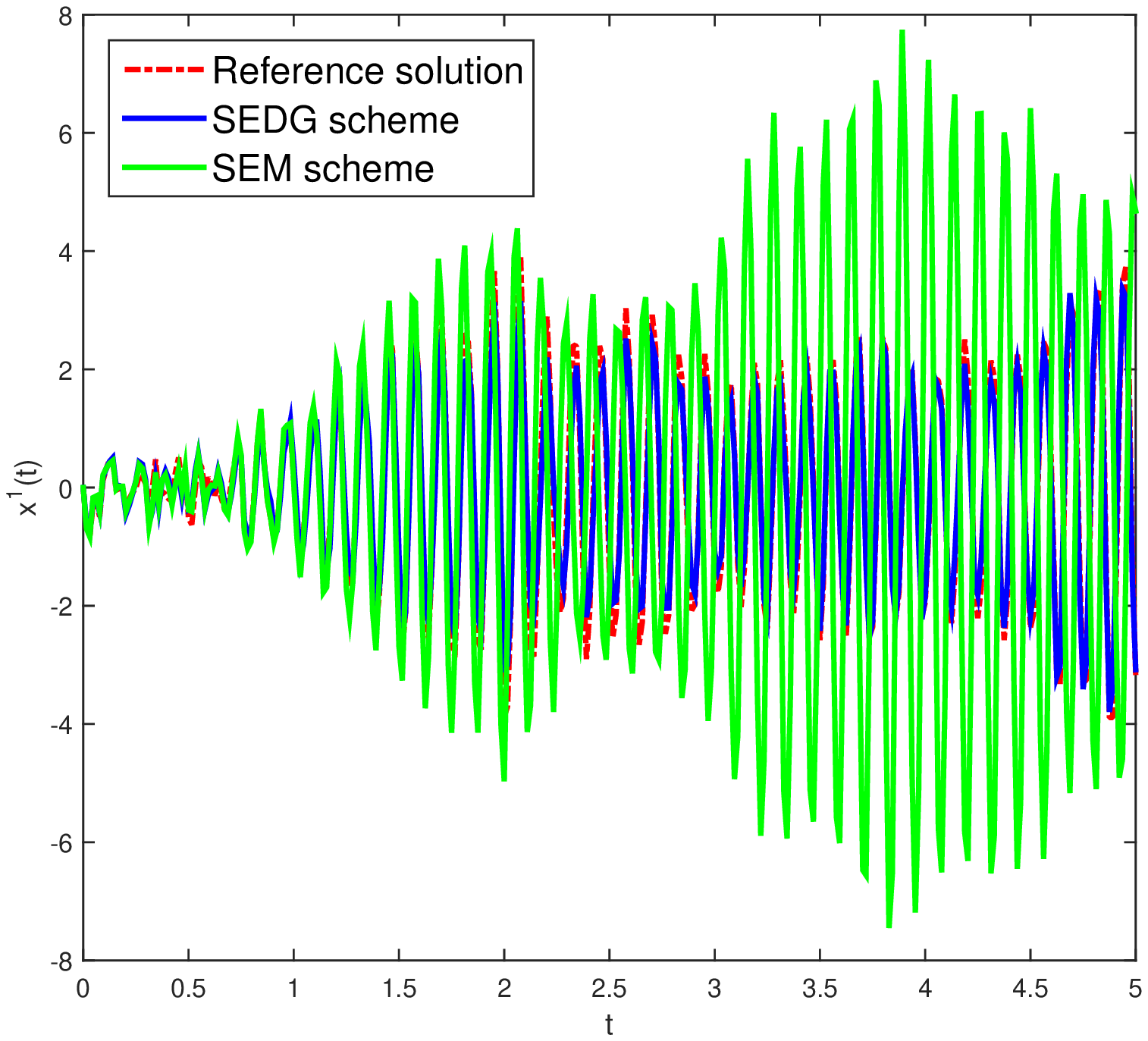}
  \label{exam-4:fig-1-b}\includegraphics[width=0.45\textwidth]{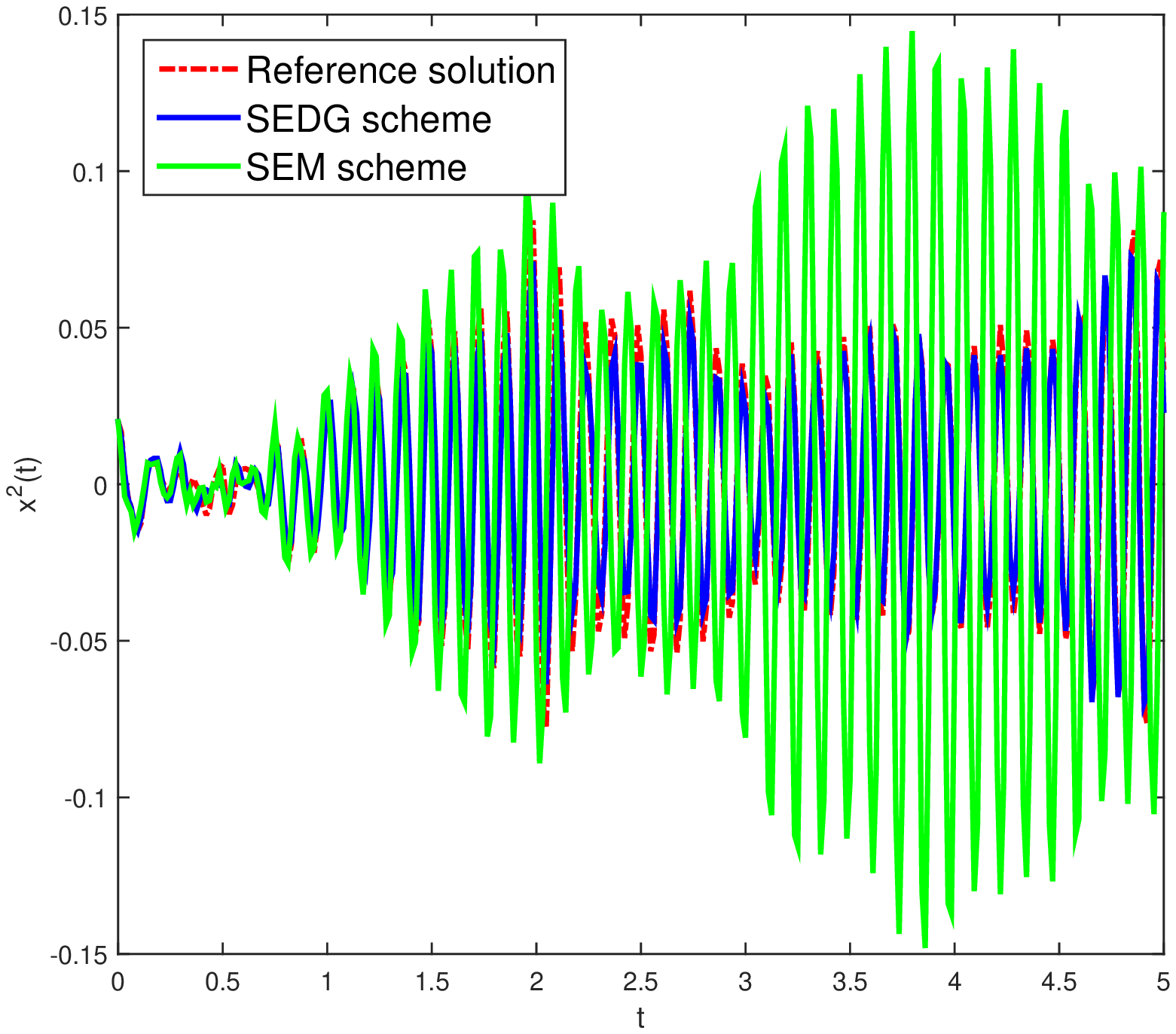}
  \caption{A sample path of $x^1(t)$ (left) and $x^2(t)$ (right) arising from the SEDG scheme (\ref{exam-4:eq-2}) (blue), the SEM scheme (\ref{addsem}) (green), and the exact solution (red). }
  \label{exam-4-fig-1}
\end{figure}
The effect of the two schemes are obvious from Fig. \ref{exam-4-fig-1}, where the sample path produced by the SEDG scheme visually coincides with that of the exact solution, while the SEM method produces large error, due to stiffness of the system, for which we take $\omega = 50$. The high frequency of the oscillation is also observed. The data setting is $\sigma = 2$, $t \in [0,5]$, $x_0 = (0,0.02)^{\mathrm{T}}$, and $h = 2^{-6}$.

\begin{figure}[htbp]
  \centering
  \subfloat[$\omega = 5$]{
   \label{exam-4:fig-2-a}\includegraphics[width=0.45\textwidth]{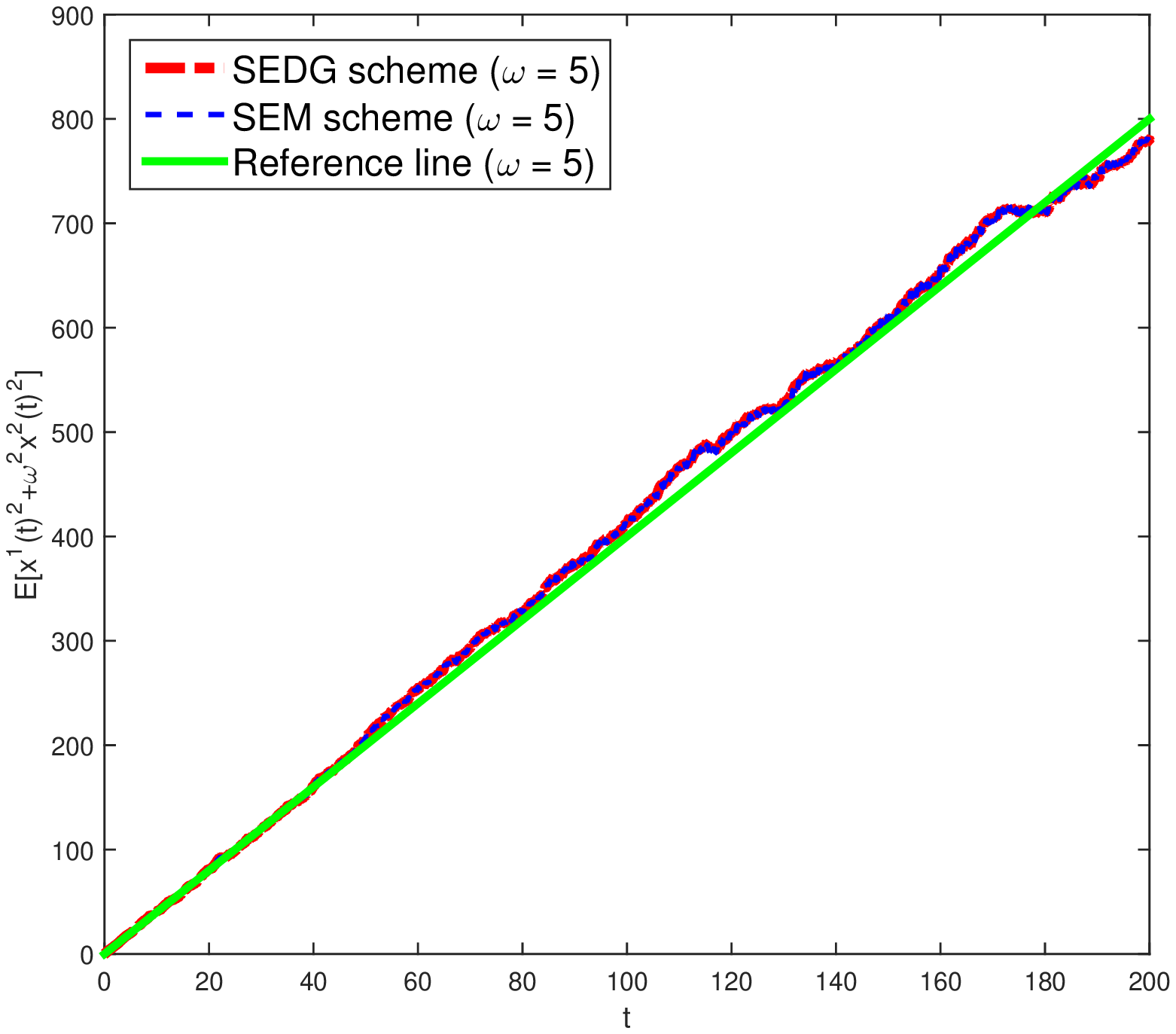}}
  \subfloat[$\omega = 50$]{
   \label{exam-4:fig-2-b}\includegraphics[width=0.45\textwidth]{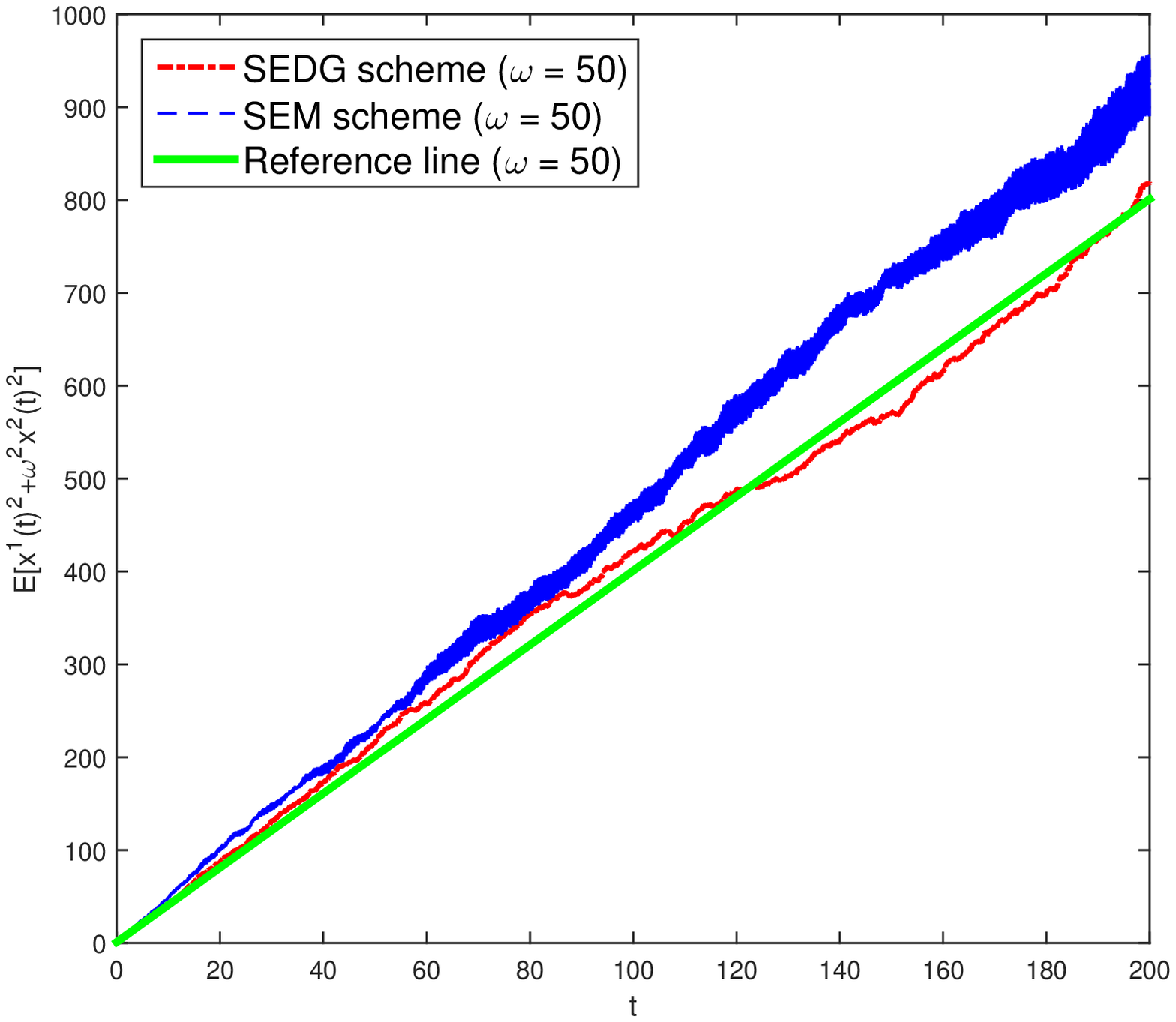}}
  \caption{Linear growth of $E(H_1)$ by the SEDG scheme (\ref{exam-4:eq-2}) (red),  the SEM scheme (\ref{addsem}) (blue), and the exact solution (green) when $\omega = 5$ (left) and $\omega = 50$ (right).}
  \label{exam-4-fig-2}
\end{figure}

For the lower frequency $\omega=5$, both the SEDG scheme and SEM scheme can well preserve the linear growth of $E(H_1)$, as illustrated by the left panel of Fig. \ref{exam-4-fig-2} . For a larger frequency $\omega=50$, however, the symplectic method fails to preserve the property with a reasonable growth rate,  while the SEDG scheme still behaves fairly well, as can be seen from the right panel. The data setting here is the same with that for Fig. \ref{exam-4-fig-1}.

 %\begin{figure}[htbp]
  %\centering
   %\includegraphics[width=0.9\textwidth,height=0.45\textwidth]{edg4-3-c.eps}
 % \caption{Root mean-square orders of the SEDG scheme \ref{exam-4:eq-2}  and the SEM scheme \ref{addsem}  for the system \ref{exam-4:eq-1}, by $\omega=5$ and $\omega=20$.}
  %\label{exam-4-fig-3}
%\end{figure}

%In \ref{exam-4-fig-3}, two different frequency parameters $\omega = 5$ and $\omega = 20$ are taken to verify root mean-square convergence orders of the SEDG scheme and SEM scheme applied to the system \ref{exam-4:eq-1}.  We observe that the two methods are both of root mean-square order 1. The errors are computed at $T = 1$, and we choose $\sigma = 2$, $x_0 = (0,0.02)^{\mathrm{T}}$. 1000 trajectories are sampled for approximating the expectation.

%------------------------------------------------

% subsection 2
%------------------------------------------------
\subsection{A stochastic Poisson system}
Let us consider the following SDE with a multiplicative noise
\begin{equation}\label{exam-1:eq-1}
{\rm d} \begin{pmatrix}
      x^1 \\
      x^2
      \end{pmatrix} ={} \begin{pmatrix}
      0 & -1 \\
      1 & 0
      \end{pmatrix} \left[ \begin{pmatrix}
      \lambda & 0 \\
      0 & \lambda
      \end{pmatrix} \begin{pmatrix}
      x^1 \\
      x^2
      \end{pmatrix} + \begin{pmatrix}
      \frac{1}{2}(({x^1})^2 - {(x^2})^2) \\
      - x^1 x^2
      \end{pmatrix} \right] ( {\rm d}t + \sigma \circ {\rm d}W(t) ),
\end{equation}
where $\sigma \ge 0$ is a constant.
When $\sigma = 0$, (\ref{exam-1:eq-1}) is called an averaged system in wind-induced oscillation,
with $\lambda$ being a detuning parameter
(see e.g.,  \cite{guckenheimer2013nonlinear}, \cite{melbo2004numerical}). The system (\ref{exam-1:eq-1}) can be of the form of (\ref{intro:poisson}), where
\begin{equation}\label{exam-1:eq-2}
Q ={}  - J, \quad M = \lambda I, \quad U ={}  - \frac{1}{2} ( x^1 ({x^2})^2 - \frac{1}{3} ({x^1})^3 ).
\end{equation}

In our experiment, we choose $\lambda = 1$, $\sigma > 0$.  Then the invariant energy is
\begin{equation*}
\mathcal{H}=\frac{({x^1})^2 + ({x^2})^2}{2} - \frac{x^1 ({x^2})^2 - \frac{1}{3} ({x^1})^3}{2},
\end{equation*}
and the SEDG scheme (\ref{struct:eq-3}) for (\ref{exam-1:eq-1}) is
\begin{equation}\label{exam-1:eq-3}
\begin{split}
\begin{pmatrix}
      x^1_{n+1} \\
      x^2_{n+1}
      \end{pmatrix}
      & ={} \begin{pmatrix}
      \cos(h + \sigma \Delta W_n) & -\sin(h + \sigma \Delta W_n) \\
      \sin(h + \sigma \Delta W_n) & \cos(h + \sigma \Delta W_n)
      \end{pmatrix} \begin{pmatrix}
      x^1_{n} \\
      x^2_{n}
      \end{pmatrix} \\
      & + \begin{pmatrix}
      \cos(h + \sigma \Delta W_n)-1 & -\sin(h + \sigma \Delta W_n) \\
      \sin(h + \sigma \Delta W_n) & \cos(h + \sigma \Delta W_n)-1
      \end{pmatrix}\\
      &\,\,\,\,\begin{pmatrix}
      \frac{1}{6}(({x^1_n})^2 + x^1_n x^1_{n+1} + ({x^1_{n+1}})^2) - \frac{1}{4} (({x^2_n})^2 + ({x^2_{n+1}})^2) \\
      -\frac{1}{4} (x^1_{n}+x^1_{n+1}) (({x^2_n})^2 + ({x^2_{n+1}})^2)
      \end{pmatrix}.
\end{split}
\end{equation}

\begin{figure}[htbp]
  \centering
  \label{exam-1:fig-1-b}\includegraphics[width=0.45\textwidth]{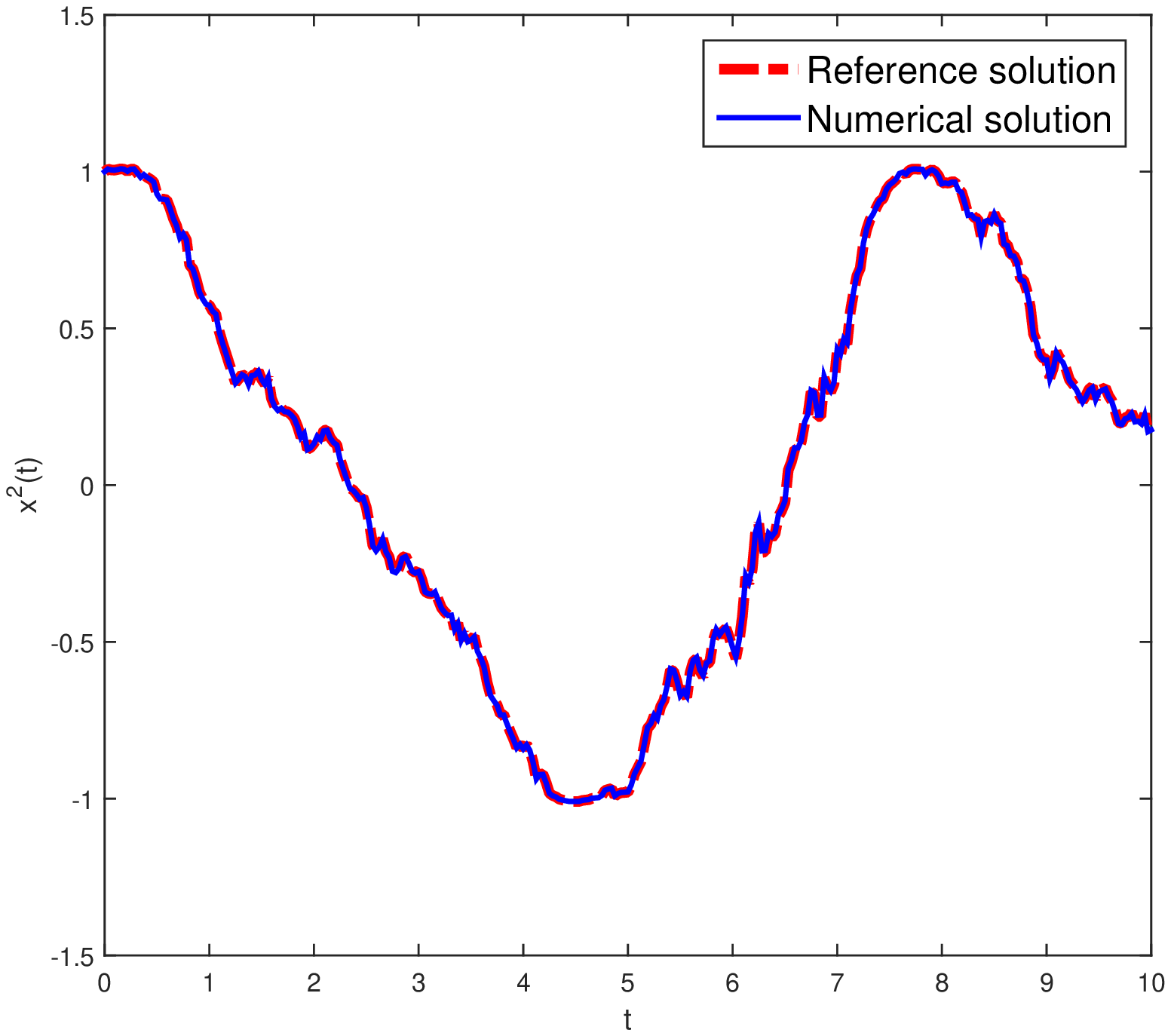}
  \label{exam-1:fig-1-c}\includegraphics[width=0.45\textwidth]{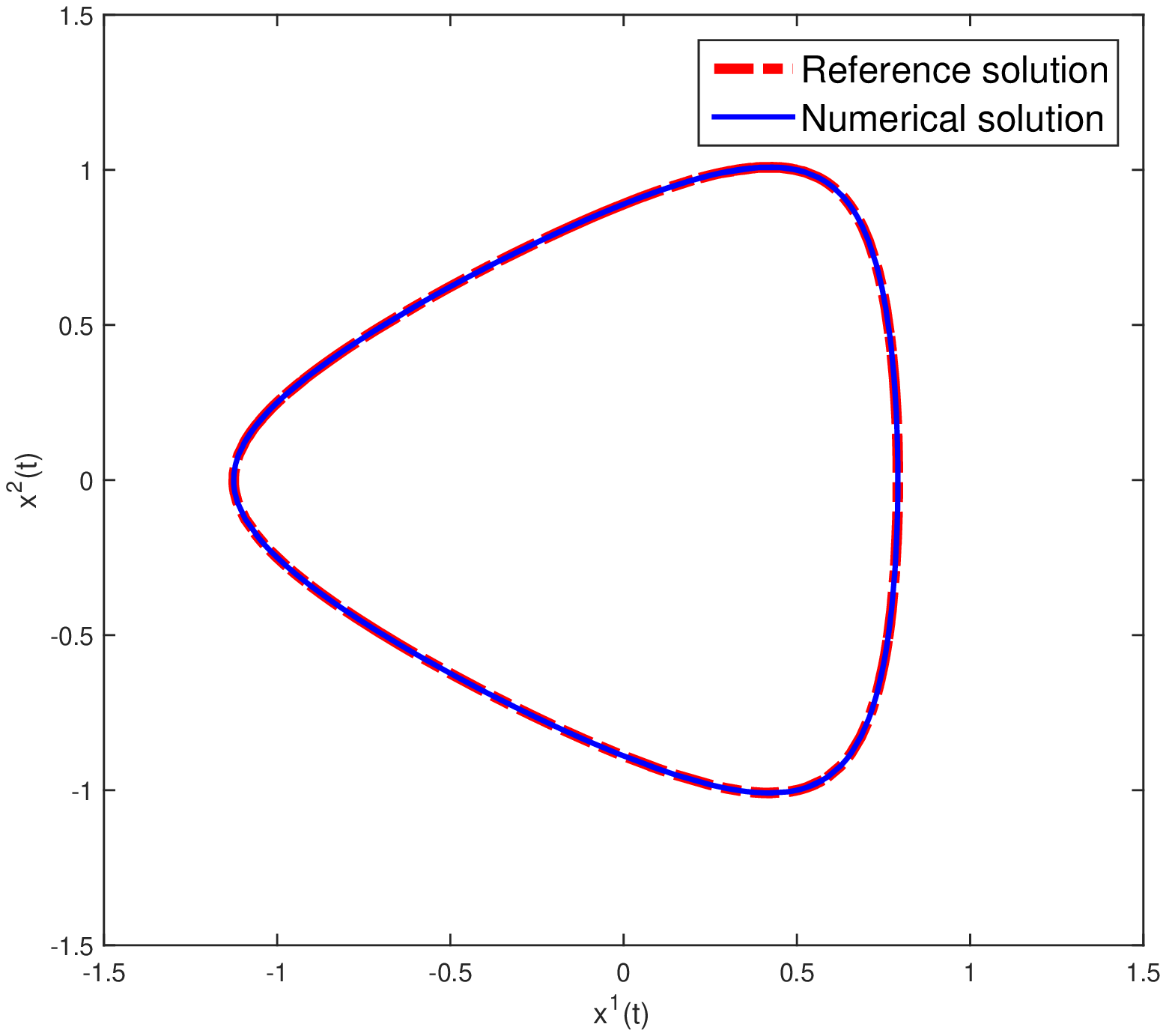}
  \caption{A sample path of $x^2(t)$  (left),  and a sample phase trajectory (right) produced by the scheme (\ref{exam-1:eq-3}) for the system (\ref{exam-1:eq-1}). }
  \label{exam-1-fig-1}
\end{figure}

Fig. \ref{exam-1-fig-1} compares the numerical sample paths arising from the scheme (\ref{exam-1:eq-3}) with the reference exact solution. Good coincidence is observed, showing the accuracy of the method. Here we choose $t \in [0,10]$, $\sigma = 0.3$, $x_0 = (0.1,1.0)^{\mathrm{T}}$ and the step size $h = 2^{-5}$.

\begin{figure}[htbp]
  \centering
  \includegraphics[width=0.9\textwidth,height=0.45\textwidth]{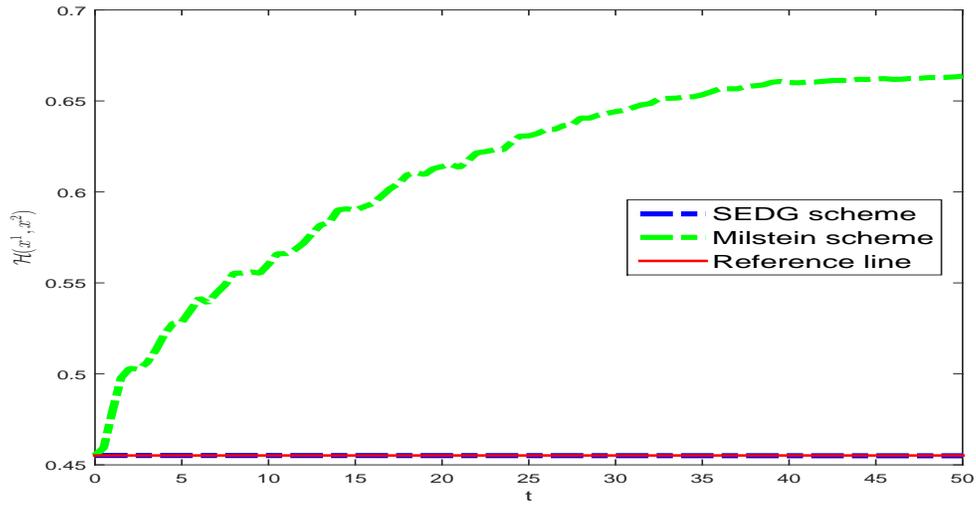}
  \caption{Preservation of the energy $\mathcal{H}$ by the exact solution (red), the SEDG method (\ref{exam-1:eq-3}) (blue), and the Milstein scheme (green).}
  \label{exam-1-fig-2}
\end{figure}

Fig. \ref{exam-1-fig-2} is devoted to show and compare the energy-preserving property of the exact solution, the SEDG method (\ref{exam-1:eq-3}) and the Milstein scheme. As can be seen from the figure, the SEDG scheme can well preserve the invariant quantity $\mathcal{H}$ of the exact solution, while the Milstein scheme fails to preserve $\mathcal{H}$. Here we take $t \in [0,50]$, $x_0 = (0.1,1.0)^{\mathrm{T}}$. The step size is $h = 2^{-4}$, and $\sigma = 0.3$.

\begin{figure}[htbp]
  \centering
  \includegraphics[width=0.9\textwidth,height=0.45\textwidth]{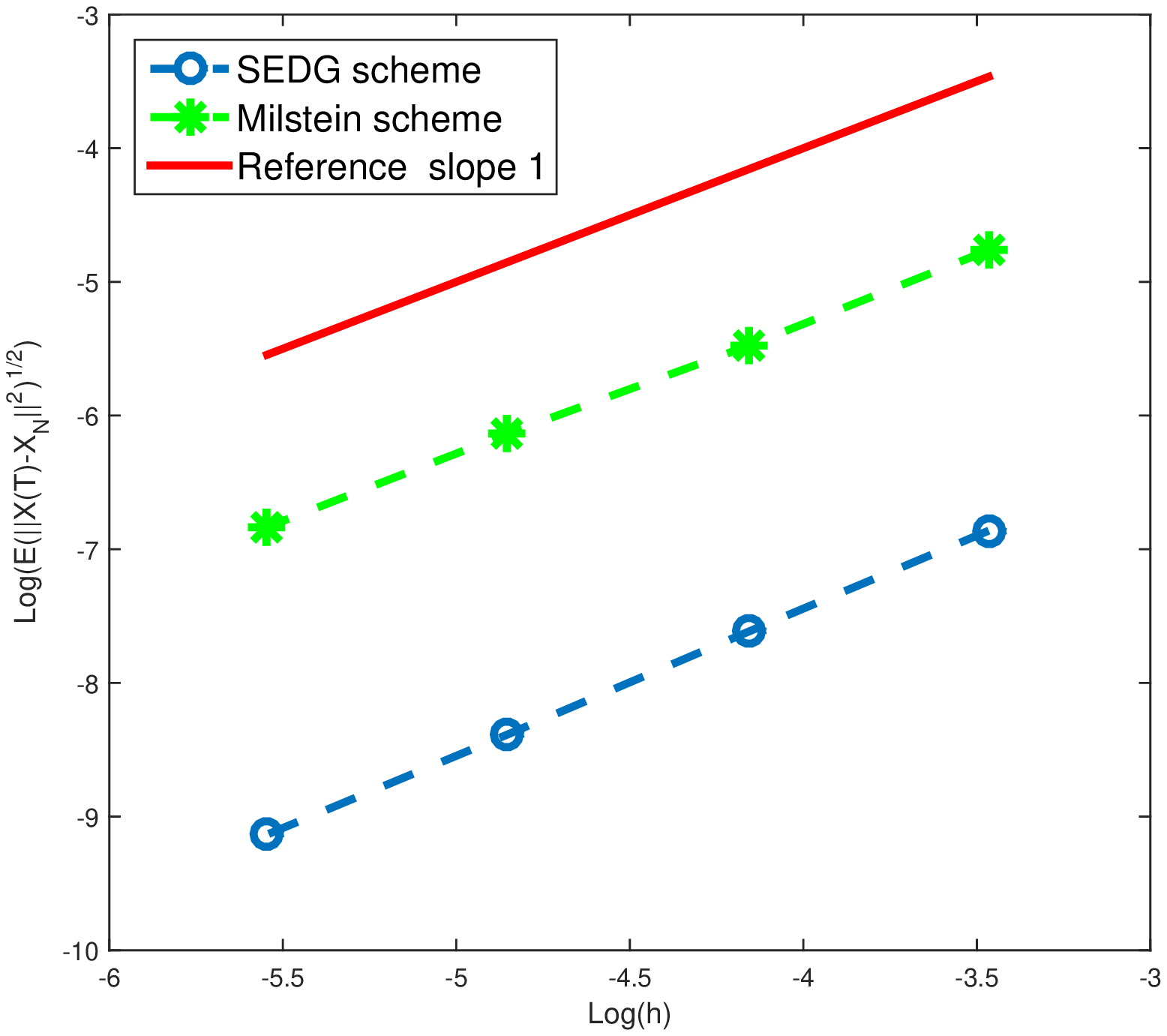}
  \caption{Root mean-square convergence orders of the SEDG scheme (\ref{exam-1:eq-3}) and the Milstein scheme in simulating (\ref{exam-1:eq-1}) . }
  \label{exam-1-fig-3}
\end{figure}

The root mean-square convergence orders of the SEDG scheme (\ref{exam-1:eq-3}) and the Milstein scheme for simulating the system (\ref{exam-1:eq-1}) are illustrated in Fig. \ref{exam-1-fig-3}, which are both 1, while obviously the error of our SEDG scheme is smaller that that of the Milstein scheme.  Here we take $x_0 = (0.1,1.0)^{\mathrm{T}}$, $\sigma = 0.3$, and calculate the error at $T = 1$. For approximating the expectation 1000 trajectories are sampled.
%------------------------------------------------
% subsection 3
%------------------------------------------------
\subsection{A linear oscillator with damping}
Consider the $2$-dimensional stochastic Langevin-type equation
\begin{equation}\label{exam-2:eq-1}
{\rm d} \begin{pmatrix}
      x^1 \\
      x^2
      \end{pmatrix} ={} \begin{pmatrix}
      - \nu & 0 \\
      1 & 0
      \end{pmatrix} \begin{pmatrix}
      x^1 \\
      x^2
      \end{pmatrix} {\rm d}t + \begin{pmatrix}
      0 & -1 \\
      1 & 0
      \end{pmatrix} \begin{pmatrix}
      0 \\
      x^2
      \end{pmatrix} {\rm d}t + \begin{pmatrix}
      \sigma \\
      0
      \end{pmatrix} \circ {\rm d}W(t),
\end{equation}
where $\nu \ge 0$ and $\sigma \not = 0$ are constants. It is the equation (\ref{lang:eq-1}) when
$M = I$, $U_0 ={} \frac{1}{2} ({x^2})^2$.

According to (\ref{lang:eq-5}), the SEDG scheme for (\ref{exam-2:eq-1}) is
\begin{equation}\label{exam-2:eq-2}
\begin{split}
 x^1_{n+1} ={} & \exp(- \nu h) x^1_{n} - \frac{1-\exp(- \nu h)}{\nu} \frac{x^2_n,x^2_{n+1}}{2}
 	     +  \exp(- \frac{\nu h}{2}) \sigma \Delta W_n, \\
 x^2_{n+1} ={} & Q_{n} + \frac{1-\exp(- \nu h)}{\nu} x^1_{n}
 	      +  \frac{1}{\nu} \left( \frac{1-\exp(- \nu h)}{\nu} - h \right) \frac{x^2_n,x^2_{n+1}}{2} \\
 	      + & \frac{1-\exp(- \frac{\nu h}{2})}{\nu} \sigma \Delta W_n.
\end{split}
\end{equation}
\begin{figure}[htbp]
  \centering
  \label{exam-2:fig-1-b}\includegraphics[width=0.45\textwidth]{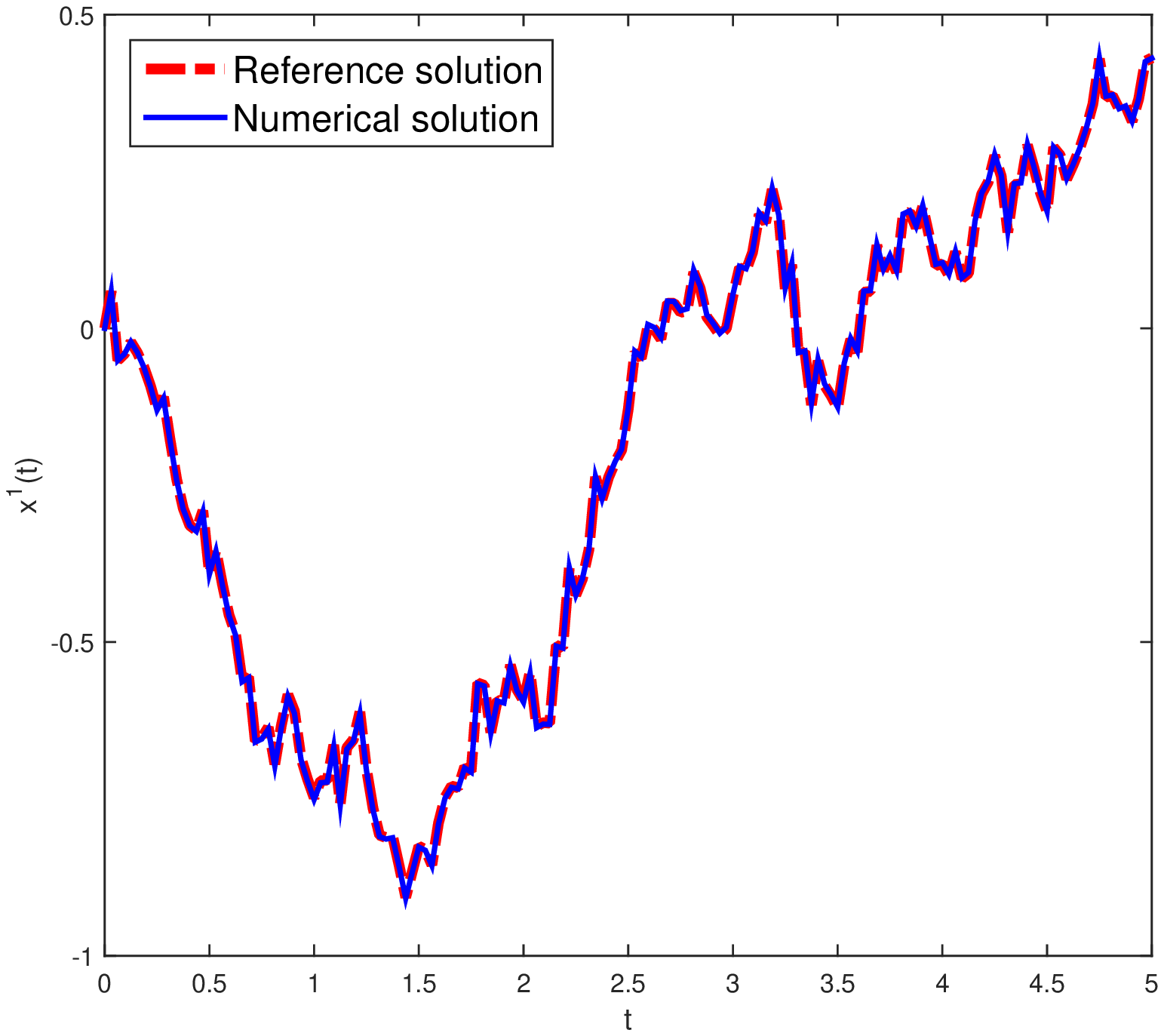}
  \label{exam-2:fig-1-c}\includegraphics[width=0.45\textwidth]{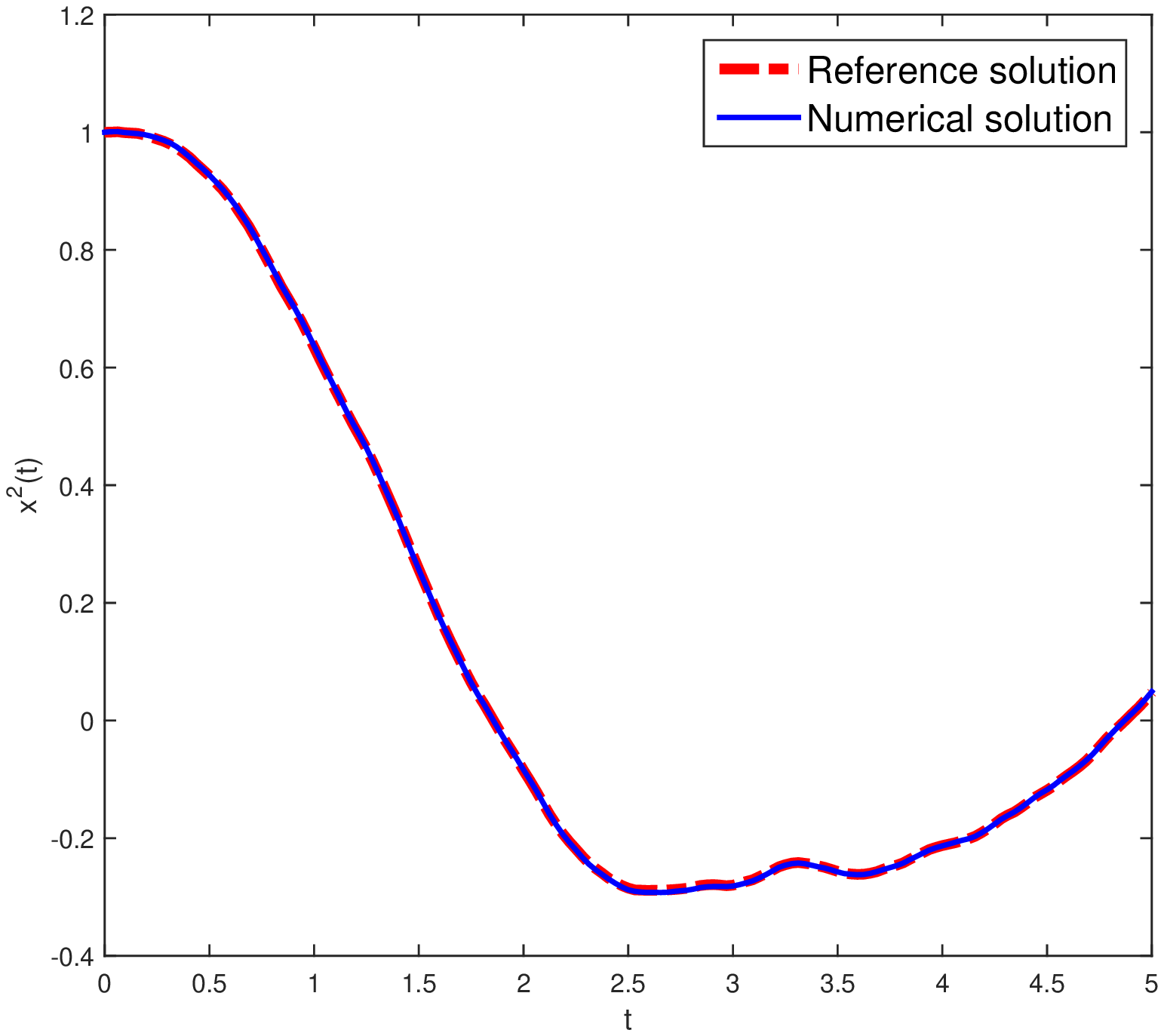}
  \caption{A sample path of $x^1(t)$ (left) and $x^2(t)$ (right) produced by (\ref{exam-2:eq-2}) for the system (\ref{exam-2:eq-1}). }
  \label{exam-2-fig-1}
\end{figure}

Fig. \ref{exam-2-fig-1} shows the path-wise simulation effect of the SEDG scheme (\ref{exam-2:eq-2}). The numerical sample paths of both components $x^i(t)$ ($i=1,2$) are visually coincident with the corresponding reference exact solution curves. We take $t \in [0,10]$, $\sigma = 0.3$ and $\nu = 1$. The initial value is $x_0 = (0,1)^{\mathrm{T}}$, and the step size is $h = 2^{-5}$.
\begin{figure}[htbp]
  \centering
  \subfloat[$\nu = 1$]{
   \label{exam-2:fig-2-a}\includegraphics[width=0.45\textwidth]{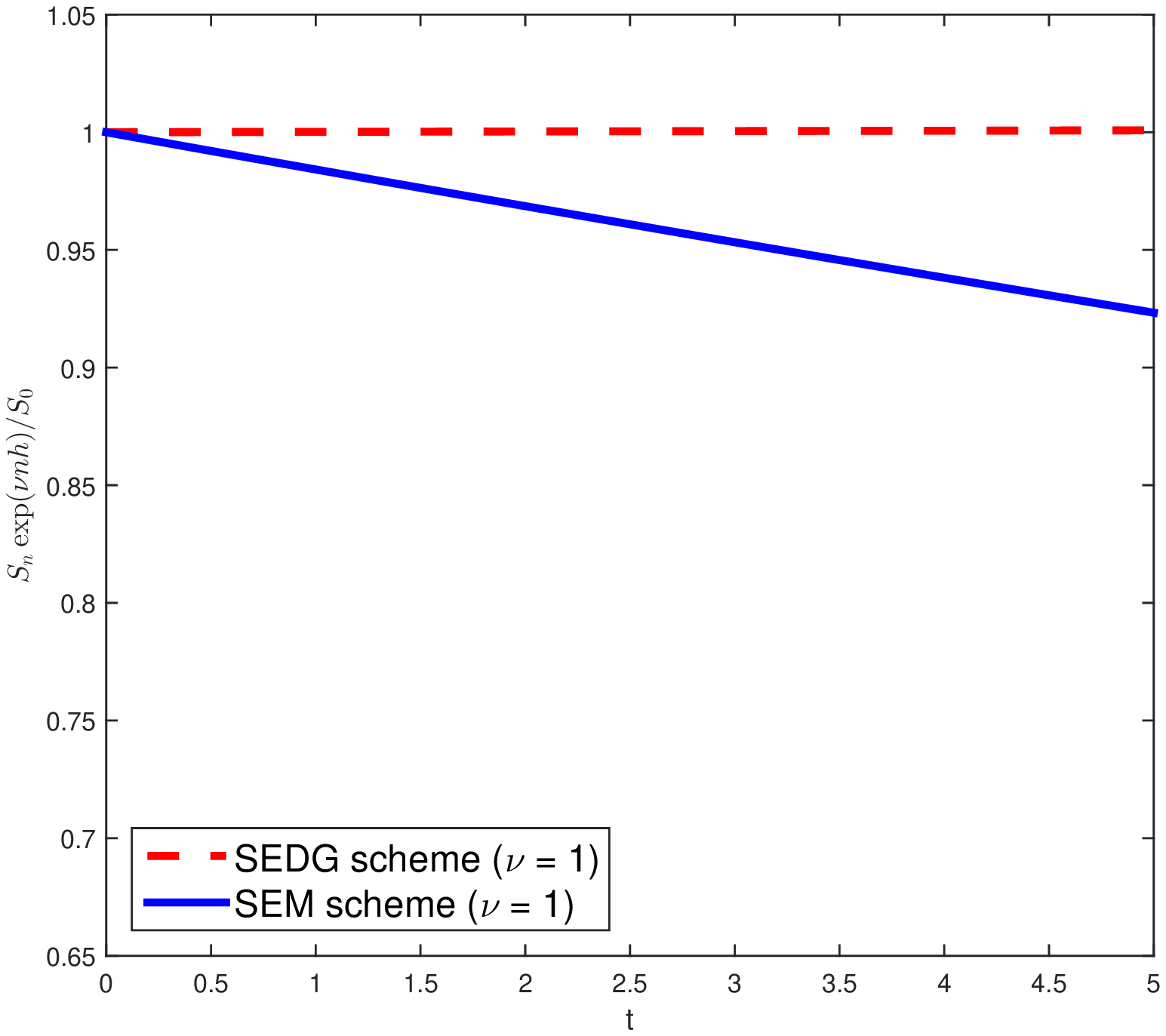}}
  \subfloat[$\nu = 2$]{
   \label{exam-2:fig-2-b}\includegraphics[width=0.45\textwidth]{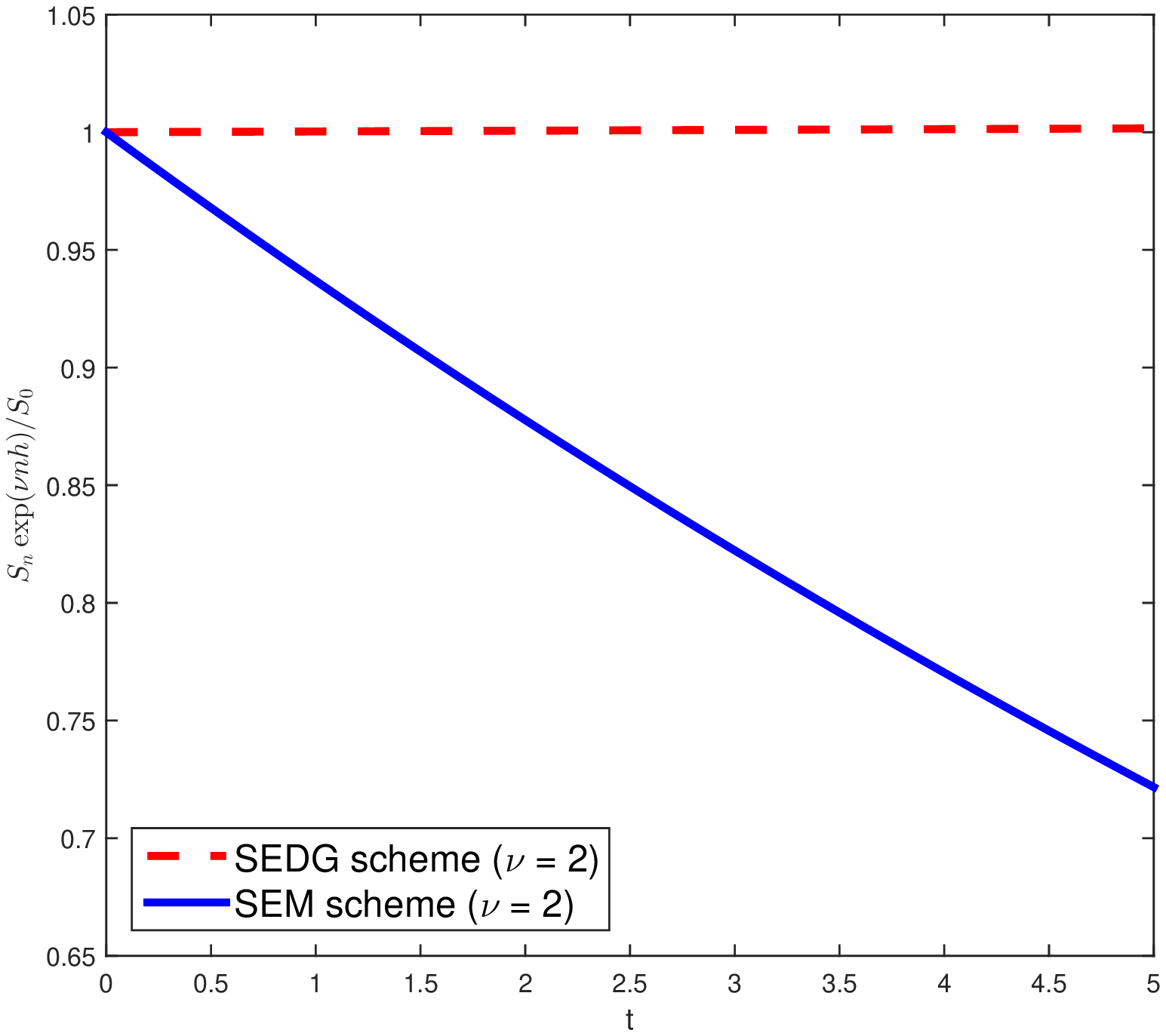}}
  \caption{Evolution of the quantity $\frac{S_n \exp(\nu t_n)}{S_0}$ arising from the SEDG scheme (\ref{exam-2:eq-2}) and the SEM scheme, with parameter $\nu = 1$ (left) and $\nu = 2$ (right).}
  \label{exam-2-fig-2}
\end{figure}

Geometrically, the conformal symplecticity of the system (\ref{exam-2:eq-1}) implies that, the area of a initial triangle $S_0$ in the phase space should decay exponentially along the flow with the evolution of time, that is, $S_n=\exp(-\nu t_n)S_0$, where $S_n$ denotes the area of the triangle at time $t_n$. In other words, the value $\frac{S_n \exp(\nu t_n)}{S_0}$ should remain at 1 along the exact flow. Fig. \ref{exam-2-fig-2} illustrates the evolution of the quantity $\frac{S_n \exp(\nu t_n)}{S_0}$ along the numerical flows produced by the SEDG scheme (\ref{exam-2:eq-2}) and by the symplectic Euler-Maruyama (SEM) scheme, for $\nu=1$ and $\nu=2$, respectively, on the time interval $[0,5]$. It can be seen that, the SEDG scheme (\ref{exam-2:eq-2}) can preserve the conformal symplecticity with high accuracy, while the SEM fails to preserve this structure. The points for the initial triangle are $x_0 = (-1,0)^{\mathrm{T}}, (0,1)^{\mathrm{T}}, (1,0)^{\mathrm{T}}$, and we take $\sigma = 0.3$, $h = 2^{-5}$.

\begin{figure}[htbp]
  \centering
  \subfloat[$\nu = 1$]{
   \label{exam-2:fig-3-a}\includegraphics[width=0.45\textwidth]{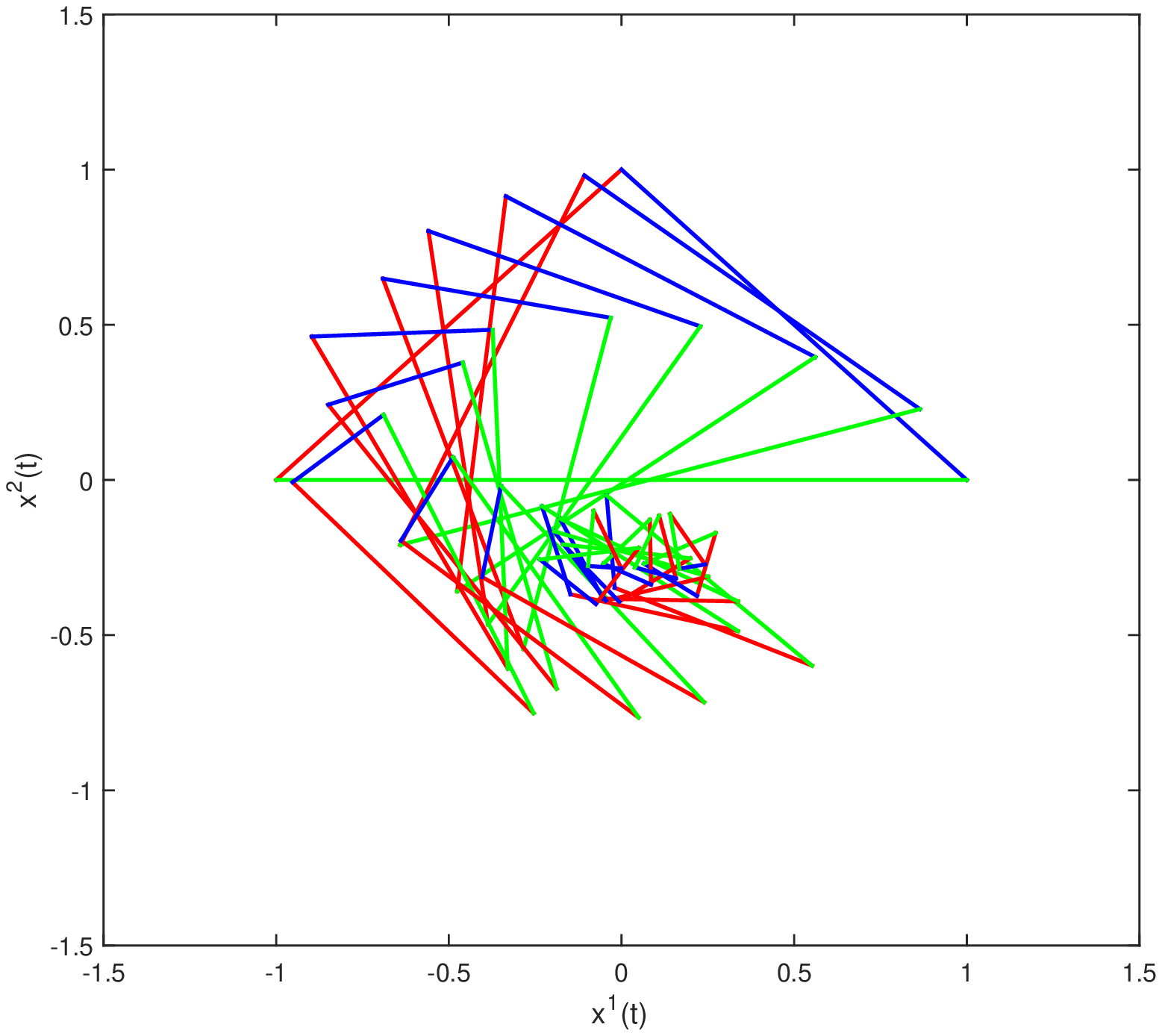}}
  \subfloat[$\nu = 2$]{
   \label{exam-2:fig-3-b}\includegraphics[width=0.45\textwidth]{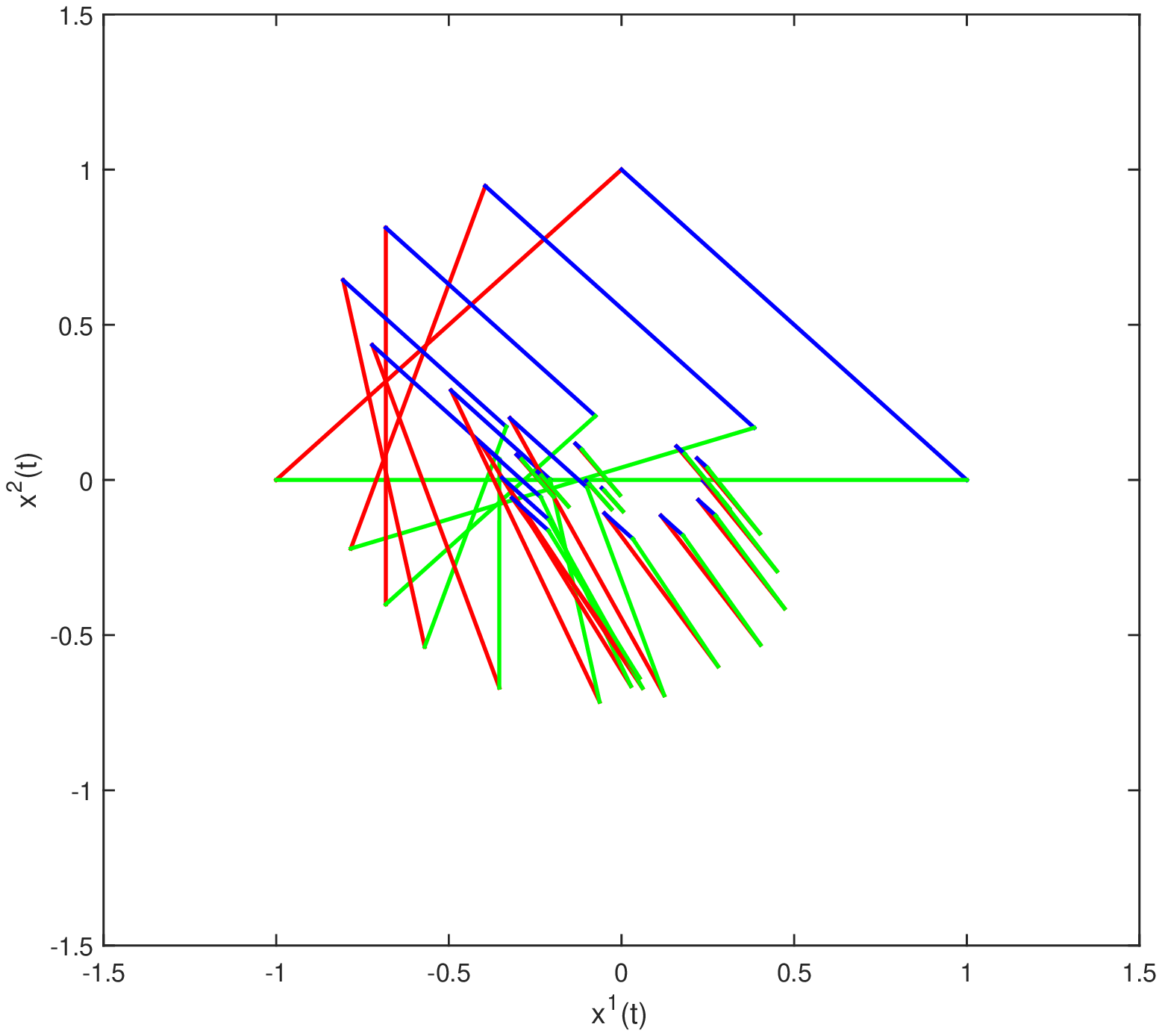}}
  \caption{Evolution of the numerical triangles arising from the SEDG scheme (\ref{exam-2:eq-2}) in the phase space for $\nu=1$ (left) and $\nu=2$ (right). }
  \label{exam-2-fig-3}
\end{figure}
Fig. \ref{exam-2-fig-3} shows the change of the triangles produced by the SEDG scheme (\ref{exam-2:eq-2}) in the phase space. The decay of the areas can be seen clearly, and with the growth of $\nu$ the decay become seemingly faster.  The parameters are the same with those for Fig. \ref{exam-2-fig-2}.

\begin{figure}[htbp]
  \centering
   \includegraphics[width=0.9\textwidth,height=0.45\textwidth]{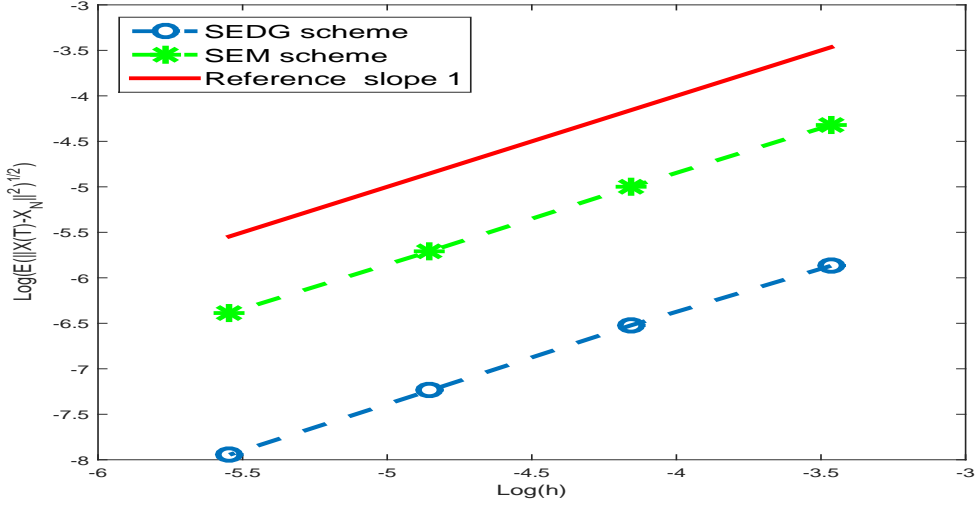}
  \caption{Root mean-square orders of the SEDG scheme (\ref{exam-2:eq-2})  and the SEM scheme for the system (\ref{exam-2:eq-1}).}
  \label{exam-2-fig-4}
\end{figure}

Fig. \ref{exam-2-fig-4} indicates that both the SEDG scheme (\ref{exam-2:eq-2}) and the SEM scheme for the system (\ref{exam-2:eq-1}) have root mean-square convergence order 1, while the error of our SEDG scheme is smaller than that of the SEM scheme. We take $\sigma = 0.3$, and calculate the error at $T = 1$. The initial value is $x_0 = (0,1)^{\mathrm{T}}$, and 1000 trajectories are sampled for approximating the expectation.
%-----------------
% subsection 4
%-----------------------------------------------------------
%------------------------------------------------

% section 5
%------------------------------------------------

\section{Conclusion}
\label{sec:concl}
For SDEs with linear and gradient components in the coefficients, namely the L-G SDEs, we proposed a class of stochastic exponential discrete gradient (SEDG) schemes, and investigated their performance in terms of accuracy and structure-preservation. Theoretical and experimental analysis showed the effectiveness and efficiency of the SEDG schemes. In particular, we demonstrated that the combination of the exponential integrator with the discrete gradient method in our scheme enables us to simulate certain stochastic highly oscillatory systems with satisfactory accuracy, large step sizes, and preservation of certain geometric structures of the systems.
\appendix
\section{Proof of Theorem \ref{struct:thm-3}}
\begin{proof}
  Similar to the proof of  Theorem \ref{const:thm-1}, we need to calculate the $p_1, p_2$ satisfying (\cite{milstein1994numerical})
  \begin{equation}\label{struct-thm-3:eq-1}
   \begin{split}
    | E(X(t_n+h) - X_{n+1}) | & ={} O(h^{p_1}), \\
    (E{| X(t_n+h) - X_{n+1} |}^2)^{\frac{1}{2}} & ={} O(h^{p_2}).
   \end{split}
  \end{equation}

  By (\ref{struct:eq-2})and (\ref{struct:eq-3}), we obtain (suppose $X(t_n) = X_n$)
  \begin{equation}\label{struct-thm-3:eq-2}
    X(t_n+h) - X_{n+1} ={} \bar{P}_{1} + \bar{P}_{2}
  \end{equation}
  where
  \begin{equation*}
   \begin{split}
    \bar{P}_{1} & ={} \int_{t_n}^{t_{n+1}} E_A (s) Q [\nabla U(X(s)) - \bar{\nabla} U(X_n,X_{n+1})] ({\rm d}s + \sigma \circ {\rm d}W(s) ), \\
    \bar{P}_{2} & ={} \left( \int_{t_n}^{t_{n+1}} E_A (s) ({\rm d}s + \sigma \circ {\rm d}W(s) ) - A^{-1}(E_A (t_n) - I) \right) Q \bar{\nabla} U(X_n,X_{n+1}).
   \end{split}
  \end{equation*}

  Further, denote $g=Q Hess(U) (A X + Q \nabla U)$, we have
  \begin{equation}\label{struct-thm-3:eq-3}
  \bar{P}_{1}={} \bar{P}_{3} + \bar{P}_{4} + \bar{P}_{5},
  \end{equation}
  where
  \begin{equation*}
   \begin{split}
   \bar{P}_{3} & ={} (I -\frac{\sigma^2}{2} A) \int_{t_n}^{t_{n+1}} E_A (s) Q [\nabla U(X(s)) - \bar{\nabla} U(X_n,X_{n+1})] {\rm d}s, \\
   \bar{P}_{4} & ={} \sigma \int_{t_n}^{t_{n+1}} E_A (s) Q [\nabla U(X(s)) - \nabla U(X_n)] {\rm d}W(s), \\
   \bar{P}_{5} & ={} \sigma \int_{t_n}^{t_{n+1}} E_A (s) Q [\nabla U(X_n) - \bar{\nabla} U(X_n,X_{n+1})] {\rm d}W(s) \\
   		  & + \frac{\sigma}{2} \int_{t_n}^{t_{n+1}} E_A (s) g(X(s)) {\rm d}s.
   \end{split}
  \end{equation*}
  \begin{equation}\label{struct-thm-3:eq-4}
  \begin{split}
  \bar{P}_{2} ={} & \left( \int_{t_n}^{t_{n+1}} E_A (s) ({\rm d}s + \sigma \circ {\rm d}W(s) ) - A^{-1}(\exp( A h) - I)\right. \\
		 & \left. - \exp(A h) A^{-1}(\exp(\sigma A \Delta W_n ) - I) \right) Q \bar{\nabla} U(X_n,X_{n+1}) \\
  		 ={} & \int_{t_n}^{t_{n+1}} \exp(A (t_{n+1} - s)) \left( \exp(\sigma A ( W(t_{n+1}) - W(s) )) - I \right) {\rm d}s \\
  		 +{} & \sigma \int_{t_n}^{t_{n+1}} \exp(\sigma A ( W(t_{n+1}) - W(s) )) \left( \exp(A (t_{n+1} - s)) - \exp(A h) \right)  {\rm d}W(s) \\
  		 +{} & \frac{\sigma}{2} A \int_{t_n}^{t_{n+1}} \exp(\sigma A ( W(t_{n+1}) - W(s) )) \left( \exp(A h) - \exp(A (t_{n+1} - s)) \right) {\rm d}s \\
  		 =:{} & \bar{P}_{6} + \bar{P}_{7} + \bar{P}_{8}.
  \end{split}
  \end{equation}
 \begin{equation}\label{struct-thm-3:eq-5}
  \begin{split}
   \bar{P}_{3} & ={} (I -\frac{\sigma^2}{2} A) \int_{t_n}^{t_{n+1}} E_A (s) Q [\nabla U(X(s)) - \nabla U(X_n)] {\rm d}s \\
   		  & + (I -\frac{\sigma^2}{2} A) \int_{t_n}^{t_{n+1}} E_A (s) Q [\nabla U(X_n) - \bar{\nabla} U(X_n,X_{n+1})] {\rm d}s \\
   		  & =:{} \bar{P}_{9} + \bar{P}_{10}.
  \end{split}
  \end{equation}
  \begin{equation}\label{struct-thm-3:eq-6}
   \bar{P}_{4} ={} \sigma \int_{t_n}^{t_{n+1}} E_A (s) g(X_n) \left( \int_{t_n}^{s} {\rm d}W(t) \right) {\rm d}W(s) + R_{\bar{P}_{4}}. 		
  \end{equation}

  We use the expansion of the SDG $\bar{\nabla} U(X_n,X_{n+1})$ to get
  \begin{equation}\label{struct-thm-3:eq-7}
  \begin{split}
   \bar{P}_{5} & ={} \frac{\sigma}{2} \int_{t_n}^{t_{n+1}} E_A (s) [g(X(s)) - g(X_n)] {\rm d}s
   		   + \frac{\sigma}{2} \int_{t_n}^{t_{n+1}} E_A (s) g(X_n) {\rm d}s \\
   		  & - \frac{\sigma}{2} \int_{t_n}^{t_{n+1}} E_A (s) g(X_n) \Delta W_n {\rm d}W(s) +  R_{\bar{P}_{5}}.
   \end{split}
  \end{equation}
  Then by the definition of matrix exponentials, we obtain
  \begin{equation}\label{struct-thm-3:eq-8}
  \bar{P}_{4} +\bar{P}_{5}= \bar{P}_{11} + \bar{P}_{12} + R,
  \end{equation}
  where $R$ is the sum of higher order terms of the matrix exponential as well as $R_{\bar{P}_{4}}$ and $\bar{R}_{P_{5}}$, and
  \begin{equation*}
   \begin{split}
    \bar{P}_{11} & ={} \frac{\sigma}{2} \int_{t_n}^{t_{n+1}} E_A (s) [g(X(s)) - g(X_n)] {\rm d}s, \\
    \bar{P}_{12} & ={} \sigma g(X_n) \left( \int_{t_n}^{t_{n+1}} \left( \int_{t_n}^{s} {\rm d}W(t) \right) {\rm d}W(s) + \frac{1}{2} ( h - (\Delta W_n)^2) \right).
   \end{split}
  \end{equation*}
  It is easy to see that $\bar{P}_{12} = 0$.

  Based on the triangular inequality and the H\"older inequality, we derive that
  \begin{equation}\label{struct-thm-3:eq-9}
  \begin{split}
  | E(X(t_n+h) - X_{n+1}) | & \leq |E \bar{P}_{6} | + |E \bar{P}_{7} | + |E\bar{P}_{8} | \\
  				 & + |E \bar{P}_{9} | + |E \bar{P}_{10} | + |E \bar{P}_{11} | + |E R |, \\
  E{| X(t_n+h) - X_{n+1} |}^2 & \leq 7( E |\bar{P}_{6}|^2 + E |\bar{P}_{7}|^2  + E |\bar{P}_{8}|^2  \\
  				 & + E| \bar{P}_{9}|^2  + E |\bar{P}_{10}|^2 + E |\bar{P}_{11}|^2 + E |R|^2 ).
  \end{split}
  \end{equation}
  By properties of the Wiener process and the assumptions in Theorem \ref{struct:thm-3}, we obtain the estimate
  \begin{equation}\label{struct-thm-3:eq-10}
   | E(X(t_n+h) - X_{n+1}) | = O(h^2).
  \end{equation}
  Again, it should be noted that the part $R$ is of higher order than other terms.

  Using the H\"older inequality and the Bunyakovsky-Schwarz inequality, under the assumption that $| (I -\frac{\sigma^2}{2} A) Q | \leq m$ for a certain $m>0$, we have
  \begin{equation}\label{struct-thm-3:eq-11}
   \bar{P}_{9} = (I -\frac{\sigma^2}{2} A) Q \int_{t_n}^{t_{n+1}} \nabla U(X(s)) - \nabla U(X_n) {\rm d}s + R_{\bar{P}_{9}},
  \end{equation}
  \begin{equation}\label{struct-thm-3:eq-12}
   \begin{split}
    |\bar{P}_{9}|^2 & \leq 2 m^2 {\left| \int_{t_n}^{t_{n+1}} \nabla U(X(s)) - \nabla U(X_n) {\rm d}s \right|}^2 + 2 |R_{\bar{P}_{9}}|^2 \\
    		       & \leq 2 h m^2 \int_{t_n}^{t_{n+1}} {|\nabla U(X(s)) - \nabla U(X_n)|}^2 {\rm d}s + 2 |R_{\bar{P}_{9}}|^2 \\
		       & \leq 2 h m^2 L^2 \int_{t_n}^{t_{n+1}} {| X(s) - X_n |}^2 {\rm d}s + 2 |R_{\bar{P}_{9}}|^2, \
   \end{split}
  \end{equation}
where $L$ is the Lipschitz constant for $\nabla U$.

 We can similarly estimate the second moment of $\bar{P}_{11}$, and the second moments of $\bar{P}_{6}$, $\bar{P}_{7}$, $\bar{P}_{8}$, $\bar{P}_{10}$ can be estimated in a straightforward way. Thus, we can obtain
  \begin{equation}\label{struct-thm-3:eq-13}
   E{| X(t_n+h) - X_{n+1} |}^2 = O(h^3).
  \end{equation}

  Finally, we conclude that $p_1 = 2$, $p_2 = \frac{3}{2}$ and the root mean-square order of the SEDG scheme (\ref{struct:eq-3}) for the system (\ref{intro:poisson}) is $p = p_2 - \frac{1}{2} = 1$.
\end{proof}

%\section*{Acknowledgments}

%\bibliographystyle{plain}
%\bibliography{references}

\end{document}